\documentclass[12pt]{amsart}
\usepackage{amsmath,amsfonts,amssymb,amsthm}
\usepackage{anysize}
\marginsize{2cm}{2cm}{2cm}{2cm}
\usepackage{graphicx}
\usepackage{color}
\usepackage[pdftex]{hyperref}
\usepackage{enumerate}

\def\u{\mathfrak{u}}

\def\g{\mathfrak{g}}
\def\h{\mathfrak{h}}

\def\a{\mathfrak{a}}
\def\q{\mathfrak{q}}
\def\s{\mathfrak s}

\def\hcx{\{J_{\alpha}\}}

\def\C{\mathbb{C}}
\def\R{\mathbb{R}}
\def\Q{\mathbb{Q}}
\def\Z{\mathbb{Z}}
\def\N{\mathbb{N}}
\def\H{\mathbb H}

\def\al{\alpha}
\def\be{\beta}
\def\ga{\gamma}

\def\ad{\operatorname{ad}}
\def\tr{\operatorname{tr}}

\def\alt{\raise1pt\hbox{$\bigwedge$}}
\def\pint{\langle \cdotp,\cdotp \rangle }
\def\la{\langle}
\def\ra{\rangle}

\theoremstyle{plain}
\newtheorem{theorem}{\bf Theorem}[section]
\newtheorem{corollary}[theorem]{\bf Corollary}
\newtheorem{proposition}[theorem]{\bf Proposition}
\newtheorem{lemma}[theorem]{\bf Lemma}

\theoremstyle{definition}

\newtheorem{example}[theorem]{\bf Example}

\theoremstyle{remark}
\newtheorem{remark}[theorem]{\bf Remark}

\newcommand{\ri}{{\rm (i)}}
\newcommand{\rii}{{\rm (ii)}}
\newcommand{\riii}{{\rm (iii)}}

\title{Hypercomplex almost abelian solvmanifolds}

\author[A. Andrada]{Adri\'an Andrada}
\email{adrian.andrada@unc.edu.ar}
\author[M. L. Barberis]{Mar\'{\i}a Laura Barberis}
\email{mlbarberis@unc.edu.ar}

\date{}
\address{FAMAF, Universidad Nacional de C\'ordoba and CIEM-CONICET, Av. Medina Allende s/n, Ciudad Universitaria, X5000HUA C\'ordoba, Argentina}

\thanks{This work was partially supported by CONICET, SECyT-UNC and ANPCyT (Argentina) and the MATHAMSUD Regional Program 21-MATH-06}


\begin{document}

\begin{abstract} We give a characterization of almost abelian Lie groups carrying left invariant hypercomplex structures and we show that the corresponding Obata connection is always flat.   We determine when such Lie groups admit HKT metrics and study the corresponding Bismut connection. We obtain the classification of hypercomplex almost abelian Lie  groups in dimension~$8$ and determine which ones admit lattices. We show that the corresponding 8-dimensional solvmanifolds are nilmanifolds or admit a flat hyper-K\"ahler metric.  Furthermore, we prove that any 8-dimensional compact flat hyper-K\"ahler manifold is a solvmanifold equipped with an invariant hyper-K\"ahler structure. 
We also construct almost abelian hypercomplex nilmanifolds and solvmanifolds in higher  dimensions.
\end{abstract}

\maketitle

\section{Introduction}
Hypercomplex manifolds are close quaternionic analogues of complex manifolds. Namely, a hypercomplex manifold is a smooth manifold $M$ equipped with a triple $\{J_1,J_2,J_3\}$ of complex structures satisfying the laws of the quaternions (see \eqref{quat} below). As a consequence, each tangent space of $M$ admits an $\H$-module structure, where $\H$ denotes the quaternions. Hence, $\dim_\R M =4n$,  $n\in \N$. Such a manifold admits a unique torsion-free connection which preserves each complex structure $J_\al$. This is known as the Obata connection \cite{Ob}, and its holonomy group is contained in the quaternionic general linear group $\operatorname{GL}(n, \H)$ (see \eqref{GL_kH}).

The classification of $4$-dimensional compact hypercomplex manifolds was given in \cite{Boy}: they are either tori,  $K3$ surfaces, or quaternionic Hopf surfaces. However, such a classification in dimension $8$ is not yet available.

Many examples of hypercomplex manifolds arise from Lie groups. Indeed, Joyce showed in \cite{joy} that any compact Lie group of dimension $4n$ admits a left invariant hypercomplex structure. Also, there are many examples of nilmanifolds (i.e., compact quotients of a simply connected nilpotent Lie group by a discrete subgroup) admitting invariant hypercomplex structures (see for instance \cite{DF0,DF1,DF,BDV}). These hypercomplex nilmanifolds have proved to be very useful when studying \textit{hyper-K\"ahler with torsion} (or HKT) geometry (see \S \ref{c&hc}). Indeed, it was shown in \cite{DF} that on a nilmanifold endowed with an invariant hypercomplex structure $\hcx$ satisfying $[J_\al x,J_\al y]=[x,y]$ for all $x,y$ and all $\al$ (i.e, $\hcx$ is \textit{abelian}), any hyperhermitian metric is automatically HKT. Later it was proved in \cite{BDV} that the converse holds: if a hypercomplex nilmanifold admits an invariant HKT metric then the hypercomplex structure is abelian. Moreover, a nilmanifold was the first example of a compact hypercomplex manifold not admitting HKT metrics (see \cite{FG}). Concerning the holonomy of the Obata connection of a $4n$-dimensional hypercomplex nilmanifold, it was shown in \cite{BDV} that it is contained in $\operatorname{SL}(n,\H)$, the commutator subgroup of $\operatorname{GL}(n, \H)$. Four-dimensional Lie groups equipped with left invariant hypercomplex structures were classified in \cite{Bar} and it was shown in \cite{Bar1} that any left invariant hyperhermitian metric on such a group is conformally hyper-K\"ahler. In the eight-dimensional case, the classification of hypercomplex nilpotent Lie groups was given in  \cite{DF1} and 
it was shown  that such a Lie group is either abelian or $2$-step nilpotent. 

On the other hand, there are not many general results about hypercomplex structures on solvmanifolds (i.e., compact quotients of a simply connected solvable Lie group by a discrete subgroup). Some of these results can be found, for instance, in \cite{BF,GT}. Also, in \cite{Pa}, hypercomplex structures admitting hyper-K\"ahler or locally conformal hyper-K\"ahler metrics are studied on almost abelian Lie groups and the  classification of such Lie groups  is provided in dimensions $4n\leq 12$. We recall that a Lie group is called almost abelian if its  Lie algebra has a codimension one abelian ideal. Almost abelian Lie groups  often arise as concrete examples with interesting geometric  properties and  have been considered by several authors  \cite{AO,BaF,CM,FT,FP1,FP,Fr,FrS}. 

Our goal is to study left invariant hypercomplex structures on almost abelian Lie groups.  The outline of the article is the following:  in \S \ref{sec2} we give the basic definitions and known results on hypercomplex structures and solvmanifolds. In Theorem \ref{characterization} of \S \ref{sec3} we characterize $4n$-dimensional almost abelian Lie algebras admitting hypercomplex structures in terms of a single matrix $A\in \mathfrak{gl}(4n-1,\R)$. Using this characterization we show that the Obata connection on any hypercomplex almost abelian Lie algebra is flat (Proposition \ref{Obata-flat}). We also determine, in \S \ref{section-HKT}, when a hypercomplex almost abelian Lie algebra admits either HKT or hyper-K\"ahler metrics (Proposition \ref{HKT}) and determine the corresponding Bismut connection (Proposition \ref{bismut}). In \S \ref{eight} we give the classification of $8$-dimensional almost abelian Lie algebras carrying hypercomplex structures (Theorem \ref{classification}) and then, using a characterization of the almost abelian Lie groups that admit lattices (\cite{Bo}), we show in \S \ref{section-lattices} that only two of the associated simply connected almost abelian Lie groups admit lattices. The corresponding solvmanifolds are nilmanifolds or  admit a flat hyper-K\"ahler metric. Moreover, we show that all 8-dimensional compact flat hyper-K\"ahler manifolds, which have been classified in \cite{Whitt}, are hypercomplex almost abelian solvmanifolds (Theorem \ref{8dim-HK}). In \S \ref{higher}, we provide several examples of higher dimensional hypercomplex almost abelian nilmanifolds
and solvmanifolds  arising from Theorem \ref{characterization}. Furthermore, using the flat Obata connection, we exhibit examples of compact solvmanifolds equipped with Clifford structures of arbitrary order $k\geq 2$. Finally,  we study  the tangent  bundle of any hypercomplex almost abelian Lie group.

\smallskip

\textit{Acknowledgements.} The authors are grateful to Leandro Cagliero, Graziano Gentili, Giulia Sarfatti and Alejandro Tolcachier for useful comments.

\medskip

\section{Preliminaries}\label{sec2}

\subsection{Complex and hypercomplex structures}\label{c&hc}

A complex structure on a differentiable manifold $M$ is an automorphism
$J$ of the tangent bundle satisfying $J^2=-I $, where $I$ is the identity endomorphism, and the 
integrability condition $N_{J}(X,Y) =0$ for all  vector fields $X,Y$ on $M,$ where $N_J$ is the Nijenhuis tensor:
\begin{equation}  N_{J}(X,Y) =
[X,Y]+J([JX,Y]+[X,JY])-[JX,JY].  \label{nijen} \end{equation}
A  Hermitian structure on $M$ is a pair $(J,g )$ of a complex structure $J$ and a Riemannian metric $g$ compatible with $J$, that is, $g (JX, JY) =g( X, Y) $ for all vector fields $X, Y\in \mathfrak{X}(M)$, or equivalently, $g (JX, Y) =-g( X, JY) $. Given a Hermitian structure $(J,g )$ on $M$, there exists a unique connection $\nabla^b$ satisfying:
\[
\nabla^bg=\nabla^bJ=0,  \;\;  \quad \qquad c(X,Y,Z):=g\left(X,T^b(Y,Z)\right) \;\text{ is a $3$-form, } 
\]
for $X,Y,Z \in  \mathfrak{X}(M)$, where $T^b$ is the torsion of $\nabla^b$ (see \cite{Bis}). This connection is called the Bismut (or Strominger) connection, which has its origin in physics \cite{St}. It is given by:
\begin{equation}\label{Bismut}
g\left(\nabla^b_XY,Z\right)=g\left(\nabla^g_XY,Z\right)+\frac12\, c(X,Y,Z), \quad X,Y,Z \in  \mathfrak{X}(M),
\end{equation}
where $\nabla^g$ is the Levi-Civita connection, and the $3$-form $c$ can be computed by $c(X,Y,Z)=d\omega(J X,J Y,J Z)$. Here, $\omega$ is the K\"ahler form associated to $(J,g )$, that is, $\omega(\cdot , \cdot )=g(J \cdot , \cdot )$. The  Hermitian structure is called strong K\"ahler with torsion (or pluriclosed) when $c$ is closed. This condition is equivalent to $dJd\omega=0$. 
  
A hypercomplex structure on $M$ is a triple of complex structures $\hcx$, $\al=1,2,3$, on $M$ satisfying the following conditions: 
\begin{eqnarray}J_1J_2&=&-J_2J_1=J_3, \hspace{2cm}  
J_1^2=J_2^2=J_3^2=-I ,  \label{quat} \\
 N_{J_{\al}}(X,Y) &=& 0   \;\;\text{ for all  } X,Y \in \mathfrak{X}(M), \;\; \al = 1,2,3. \label{integ}
\end{eqnarray}
It then follows that $M$ has a family of complex structures 
$J_{y}=y_1J_1+y_2J_2+y_3J_3$ parameterized by points $y = (y_1,y_2,y_3)$ in the unit sphere $S^2 \subset \R^3$, since the relations \eqref{quat} give that $J_y^2=-I$ and conditions \eqref{integ} imply that $N_{J_y}(X, Y)=0$ for all $X, Y\in \mathfrak{X}(M)$. It follows from \eqref{quat} that $T_pM$, for each $p\in M$, has an $\H$-module structure, where $\H$ denotes the quaternions; in particular, $\dim M \equiv 0 \; \pmod{4}$.
 
The quaternionic space ${\mathbb H}^n$ provides the standard model of a hypercomplex structure. Consider the real coordinates $(x_1, y_1, v_1, w_1, \dots, x_n, y_n, v_n,w_n)$ corresponding
to a point $(q_1, \dots, q_n)$ in ${\mathbb H}^n$, where $q_l= x_l+iy_l+jv_l+kw_l, \;
1 \leq l\leq n$.  Then
\begin{align*}
J_1 \left( \partial / \partial x_l \right)&= \partial/ \partial y_l,
&  J_1 \left(\partial / \partial v_l \right)&= \partial/ \partial w_l,
\hspace{.75cm} 1\leq l \leq n,  & J_1^2&=-I,\\
 J_2 \left( \partial / \partial x_l\right) &= \partial/ \partial v_l,
 &   J_2 \left(\partial / \partial y_l \right)&= -\partial/ \partial w_l,
\hspace{.5cm} 1\leq l \leq n,
 &J_2^2&=-I
 \end{align*}
defines a hypercomplex structure on ${\mathbb H}^n$, by setting $J_3=J_1J_2$.

Assume that a $4n$-dimensional manifold $M$ admits an atlas of charts
$\{(U_a,\varphi _a)\}$\ such that the transition functions
$\varphi _a \circ \varphi _b^{-1}:
\varphi _b (U_a \cap U_b) \rightarrow \varphi _a (U_a\cap U_b)$\
are hyperholomorphic, that is, they are holomorphic with respect to $J_\al$ for any $\al$, where  $\hcx$ is the hypercomplex structure
on ${\mathbb R}^{4n}={\mathbb H}^n$ considered above. By transferring the standard hypercomplex structure from ${\mathbb R}^{4n}$ to $M$ by means of these charts, we obtain a globally defined hypercomplex structure on $M$. A hypercomplex structure $\hcx$ on $M$ is called \textit{locally flat} when it is obtained in this way.  Moreover, it was shown in \cite{Som} that a manifold admitting a locally flat hypercomplex structure is  necessarily affine, that is, the transition  functions are restrictions of quaternionic affine maps (see also \cite{GGS}).

Given a hypercomplex structure $\hcx$ on $M$, there is a unique torsion-free connection $\nabla$ on $M$ such that $\nabla J_\al=0, \; \al =1,2,3$. It is called the Obata connection (see \cite{Ob})  and it was proved in \cite{Sol}  that it  can be computed as follows:
\begin{equation}\label{Obata}
\nabla_XY = \frac12 ([X,Y]+J_1[J_1X,Y]-J_2[X,J_2Y]+J_3[J_1X,J_2Y]), \quad X,Y \in \mathfrak{X}(M).\end{equation}
More generally, it can be verified that another expression for $\nabla$ is given by:
\begin{equation}\label{obata-alpha}
\nabla_XY = \frac12 ([X,Y]+J_\al [J_\al X,Y]-J_\be [X,J_\be Y]+J_\ga [J_\al X,J_\be Y]),
\end{equation}  
for any cyclic permutation  $(\al , \be , \ga)$  of $(1, 2, 3)$. It is known that $\hcx$
is locally flat if and only if the Obata connection $\nabla$ is flat (cf.  \cite[Theorem  11.2]{Ob}, \cite{GGS}). Using this fact, it turns out that not every hypercomplex structure is locally flat, in contrast with Newlander-Nirenberg's result for the complex case (\cite{new}). For instance, the $K3$ surfaces and the homogeneous spaces  $\operatorname{SU}(3)$ and $\operatorname{SO}(6)/\operatorname{SU}(2)$ are examples of non-affine manifolds which admit hypercomplex structures (\cite{Boy,joy}).

The holonomy group of the Obata connection, $\operatorname{Hol}(\nabla)$, is contained in the quaternionic  general linear group $\operatorname{GL}(n, \H)$ (see \eqref{GL_kH} below) since each complex structure is $\nabla$-parallel. The holonomy group of the Obata
connection is an important invariant of  hypercomplex manifolds, but it is not often determined 
explicitly.  However, it was shown in \cite{Sol} that the Obata holonomy on $\operatorname{SU}(3)$, equipped with a left invariant hypercomplex structure constructed by Joyce in \cite{joy}, coincides
with $\operatorname{GL}(2,\H)$. An important subgroup of $\operatorname{GL}(n, \H)$ is its commutator subgroup $\operatorname{SL}(n, \H)$ which appears in the Merkulov-Schwachh\"ofer list  of possible holonomy groups of a torsion-free linear connection \cite{MeS}.  The study of hypercomplex manifolds with Obata holonomy contained in $\operatorname{SL}(n,\H)$ is very active; see for instance \cite{GLV,IP,LW,LW1}.

A Clifford structure of order $k\geq 1$ on a manifold $M$ is a family of $k$ pairwise anticommuting complex structures on $M$ generating a subalgebra of dimension $2^k$ of $\operatorname{End}\, (TM)$ (see \cite{BDM,joy1}). This contains complex and hypercomplex  structures as special cases, with $k=1$ and $2$, respectively. 
These structures are known as flat Clifford structures by Moroianu and Semmelmann  in \cite{MS}. It was proved in \cite{SSTVP} that  non-abelian compact Lie groups do not admit left invariant Clifford structures of order $k\geq 3$. On the other hand, examples of such structures on non-compact Lie groups were obtained in \cite{BDM,joy1} for any $k\geq 2$.

A hyperhermitian structure on $M$ is a pair $(\hcx ,g )$ where $\hcx$ is a hypercomplex structure and $(J_{\al},g )$ is Hermitian for $\al =1,2,3$. 
An interesting subclass of hyperhermitian structures is given by hyper-K\"ahler structures \cite{Cal}, which are hyperhermitian structures such that $(J_\al , g)$ is K\"ahler for $\al =1,2,3$, that is, the  K\"ahler forms $\omega_\al$ associated to $(J_\al,g)$ are closed, $\al =1,2,3$. In this case, the Levi-Civita connection coincides with the Obata connection and its holonomy group is contained in $\operatorname{Sp}\,(n)$, where $\dim M=4n$.  Since $\operatorname{Sp}\,(n)\subset \operatorname{SU}\,(2n)$,  hyper-K\"ahler metrics are Ricci-flat. 
A less restrictive class of hyperhermitian structures are 
the so-called hyper-K\"ahler with torsion (or HKT) structures \cite{HP}. These are hyperhermitian structures satisfying $\partial\Omega=0$, where $\Omega=\omega_2+i\omega_3$ and $\partial$ is the Dolbeault differential on $(M,J_1)$, with  $\omega_\al$ as above. It is known (see \cite{GP}) that a hyperhermitian structure  $(\hcx ,g )$ on $M$ is HKT if and only if $\nabla^b_1=\nabla^b_2=\nabla^b_3$, where $\nabla^b_\al $ is the Bismut connection associated to $(J_\al ,g), \; \al =1,2,3$. 
The HKT structure is called strong or weak depending on whether the torsion $3$-form $c$ of $\nabla^b_\al $ is closed or not. We note that the class of hyper-K\"ahler manifolds is strictly contained in the class of HKT manifolds (see \cite{DF}) and this, in turn, is a proper class of hypercomplex manifolds (see \cite{FG}). 

\vspace{.3cm}
  
Let $G$ be a connected real Lie group with Lie algebra $\g$. A complex structure $J$ on $G$ is said to be left invariant if left translations by elements of $G$ are holomorphic maps. In this case $J$ is determined by the value at the identity of $G$. Thus, a left invariant complex structure on $G$ amounts to a  complex structure on its Lie algebra $\g$, that is, a real linear transformation $J$ of $\g$ satisfying $J^2 = -I$ and $N_J(x, y)=0$ for all $x, y$ in $\g.$ A Riemannian metric $g$ on $G$ is called left invariant when  left translations are isometries. Such a metric $g$ is determined by its value $g_e=\pint$ at the identity $e$ of $G$, that is, $\pint$ is a positive definite inner product on $T_e G=\g$. 

A Hermitian structure $(J,g)$ on $G$ is called left invariant when  both, $J$ and $g$, are left invariant, with analogous definitions for left invariant hypercomplex, hyperhermitian, hyper-K\"ahler or HKT structures. Given a left invariant Hermitian structure $(J,g)$ on $G$, let $J$ and $\pint$ denote the corresponding complex structure and Hermitian inner product on $\g$. We say that $(J, \pint)$ is a Hermitian structure on $\g$. In the same way, a left invariant hypercomplex, hyperhermitian, hyper-K\"ahler or HKT structure on $G$ induces the corresponding structure on~$\g$.  

 Since Ricci-flat homogeneous metrics are flat (see \cite{AK}), it follows that left invariant hyper-K\"ahler metrics are flat. For the more general class of left invariant HKT metrics, 
it was proved in  \cite{DF}  that   a hyperhermitian structure $(\hcx ,\pint )$ on $\g$ is HKT if and only if the following condition is satisfied:
\begin{align}
    \la [J_1x,J_1y],z\ra +\la [J_1y,J_1z],x\ra+ & \la [J_1z,J_1x],y\ra = \nonumber \\ 
 & = \la [J_2x,J_2y],z\ra +\la [J_2y,J_2z],x\ra +
\la J_2z,J_2x],y\ra \label{inv_hkt}\\
&  =\la [J_3x,J_3y],z\ra+\la [J_3y,J_3z],x\ra +\la J_3z,J_3x],y\ra \nonumber
\end{align}
for all  $x,y,z \in {\g}$. 

\medskip

\subsection{Almost abelian Lie groups and associated solvmanifolds}

In this article we will focus on a family of solvable Lie groups, namely, the almost abelian ones, and also on their associated solvmanifolds. 
We begin by recalling some facts about general solvmanifolds, and then move on to establish some properties of almost abelian Lie groups.

\medskip

A \textit{solvmanifold} is a compact quotient $\Gamma\backslash G$, where $G$ is a simply connected solvable Lie group and $\Gamma$ is a discrete subgroup of $G$. Such a subgroup $\Gamma$ is called a \textit{lattice} of $G$. When $G$ is nilpotent and $\Gamma\subset G$ is a lattice, the compact quotient $\Gamma\backslash G$ is known as a nilmanifold.

It follows that $\pi_1(\Gamma\backslash G)\cong \Gamma$ and  $\pi_n(\Gamma\backslash G)=0$ for $n>1$. Furthermore, solvmanifolds are determined up to diffeomorphism by their fundamental groups. In fact:

\begin{theorem}\cite[Theorem 3.6]{Rag}\label{solv-isom}
If $\Gamma_1$ and $\Gamma_2$ are lattices in simply connected solvable Lie groups 
$G_1$ and $G_2$, respectively, and $\Gamma_1$ is isomorphic to $\Gamma_2$, then $\Gamma_1 \backslash G_1$ is diffeomorphic to $\Gamma_2 \backslash G_2$.
\end{theorem}

The conclusion of the previous theorem can be strengthened when both solvable Lie groups $G_1$ and $G_2$ are completely solvable\footnote{A solvable Lie group $G$ is completely solvable if the adjoint operators $\ad_x:\g\to\g$, with $x\in \g=\operatorname{Lie}(G)$, have only real eigenvalues. In particular, nilpotent Lie groups are completely solvable.}. Indeed, this is the content of Saito's rigidity theorem:

\begin{theorem}\cite{Sai}\label{Saito}
Let $G_1$ and $G_2$ be simply connected completely solvable Lie groups and $\Gamma_1 \subset G_1, \, \Gamma_2\subset G_2$ lattices. Then every isomorphism
$f: \Gamma_1 \to \Gamma_2$ 
extends uniquely to an isomorphism of Lie groups $F: G_1 \to G_2$.
\end{theorem}

Solvmanifolds of completely solvable Lie groups have a very nice property concerning their de Rham cohomology. Indeed, Hattori \cite{Hat} proved that the natural inclusion 
\begin{equation}\label{inclusion}
   \alt^* \g^* \hookrightarrow \Omega^*(\Gamma\backslash G),
\end{equation} with $G$ completely solvable, induces an isomorphism 
\begin{equation}\label{deRham}
H^*(\g) \cong H^*_{dR}(\Gamma\backslash G).
\end{equation}
That is, the de Rham cohomology of the solvmanifold can be computed in terms of left invariant forms. In particular, $H^*_{dR}(\Gamma\backslash G)$ does not depend on the lattice $\Gamma$.
The isomorphism \eqref{deRham} was proved earlier for nilmanifolds by Nomizu \cite{Nom}. 

For general  solvmanifolds it is well known that the natural  inclusion \eqref{inclusion} induces  an injective homomorphism $H^*(\g)\hookrightarrow H^*_{dR}(\Gamma\backslash G)$; in particular, \begin{equation}\label{betti1}
    b_1(\Gamma\backslash G)\geq 1 ,
\end{equation} 
see for instance \cite[Corollary 3.11]{Bo}. 

In general, it is not easy to determine whether a given Lie group $G$ admits a lattice. A well known restriction is that if this is the case then $G$ must be unimodular (\cite{Mi}), i.e. the Haar measure on $G$ is left and right invariant, which is equivalent, when $G$ is connected, to  $\tr(\ad_x)=0$ for any $x$ in the Lie algebra $\g$ of $G$. In the nilpotent case there is a criterion to determine the existence of lattices, due to Malcev:

\begin{theorem}\cite{Mal}\label{Mal}
A simply connected nilpotent Lie group has a lattice if and only if its Lie algebra admits a basis with respect to which the structure constants are rational.
\end{theorem}

We point out that left invariant geometric structures defined on $G$ induce corresponding geometric structures on $\Gamma\backslash G$, which are called invariant. For instance, a left invariant complex structure (respectively, Riemannian metric) on $G$ induces a complex structure (respectively, Riemannnian metric) on $\Gamma\backslash G$ such that the canonical projection $G\to \Gamma\backslash G$ is a local biholomorphism (respectively, local isometry). 

\ 

We move on now to almost abelian Lie groups. A Lie group $G$ is called {almost abelian} if its Lie algebra $\g$ has a codimension one abelian ideal $\u$. Such a Lie algebra will also be called almost abelian, and it can be written as $\g= \R e_0 \ltimes  \u$ for some  $e_0\notin \u$. Since $\u$ is abelian, for some $d\in\N$ we may write $\g=\R e_0\ltimes_A \R^d$, where $A\in \mathfrak{gl}(d,\R)$ is given by $A=\ad_{e_0}|_\u$ where we identify $\u$ with $\R^d$ after some choice of basis. Accordingly, the Lie group $G$ is a semidirect product 
$G=\R\ltimes_\varphi \R^d$, where the action is given by $\varphi(t)=\exp(t A)$. Clearly an almost abelian Lie algebra is solvable, and it is completely solvable if and only if the matrix $A$ has only real eigenvalues. Moreover, it is nilpotent if and only if $A$ is nilpotent. 

Regarding the isomorphism classes of almost abelian Lie algebras, we have the following result, proved in \cite{Fr}.

\begin{lemma}\label{ad-conjugated}
Two almost abelian Lie algebras $\g_1=\R e_1 \ltimes_{A_1} \R^d$ and $\g_2=\R e_2 \ltimes_{A_2}\R^d$ are isomorphic if and only if there exists $c\neq 0$ such that $A_2$ and $cA_1$ are conjugate. 
\end{lemma}

An important feature concerning almost abelian Lie groups is that there exists a criterion to determine when such a Lie group admits lattices. Indeed, there is the following result which will prove very useful in forthcoming sections:

\begin{proposition}\label{latt}\cite{Bo}
Let $G=\R\ltimes_\varphi\R^d$ be a unimodular almost abelian Lie group. Then $G$ admits a lattice if and only if there exists  $t_0\neq 0$ such that $\varphi(t_0)$ is conjugate to an invertible integer matrix.  In this situation, a lattice is given by $\Gamma=t_0 \Z\ltimes P\mathbb Z^{d}$, where $P\in \operatorname{GL}(d,\R)$ satisfies $P^{-1}\varphi(t_0)P\in \operatorname{SL}(d,\Z)$. 
\end{proposition}

Note that if $E:=P^{-1}\varphi(t_0)P$ then $\Gamma\cong \Z\ltimes_E \Z^d$, where the group multiplication in this last group is given by 
\[ (m,(p_1,\ldots,p_d))\cdot (n,(q_1,\ldots, q_d))=(m+n,(p_1,\ldots,p_d)+E^m(q_1,\ldots, q_d)). \]

\

\section{Hypercomplex structures on almost abelian Lie algebras}\label{sec3}
In this section we will study hypercomplex structures $\hcx$ on almost abelian Lie algebras $\g=\R \ltimes_A \R^{4n-1}$. In order to do this, 
we recall from \cite[Lemma 3.1]{AO} (see also \cite[Lemma 6.1]{LRV}) the characterization of almost abelian Lie algebras with a Hermitian structure: 
\begin{lemma}\label{hermitian}
Let $\g$ be an almost abelian Lie algebra with codimension one abelian ideal $\u$, admitting  a Hermitian structure  $(J, \pint)$. Then $\a:=\u \cap J\u $ is a $J$-invariant abelian ideal  of codimension $2$. Moreover, there exist an orthonormal basis $\{f_1,f_2=Jf_1\}$ of ${\a}^{\perp}$, $v_0\in\a$ and $\mu\in\R$ such that $f_2 \in \u$, $[f_1,f_2]=\mu f_2 + v_0$ and $\ad_{f_1}|_{\a}$ commutes with $J|_{\a}$.
\end{lemma}

Given a hypercomplex structure $\hcx$ on $\R^{4k}$, we will denote by
\begin{equation}\label{GL_kH} \operatorname{GL}(k, \H):= \{T\in \operatorname{GL}(4k,\R) : TJ_\al =J_\al T \text{ for all } \al \}, 
\end{equation}
the quaternionic general linear group, with corresponding Lie algebra:
\[ \mathfrak{gl}(k,\mathbb{H})= \{T\in \mathfrak{gl}(4k,\R) : TJ_\al =J_\al T \text{ for all } \al \}. \]

\

The characterization of almost abelian Lie algebras admitting a hypercomplex structure is given in the next result.

\begin{theorem}\label{characterization}
Let $\g$ be a $4n$-dimensional almost abelian Lie algebra with codimension one abelian ideal $\u$,  admitting a hypercomplex structure $\hcx$. 
Let $\h := \u \cap J_1\u\cap J_2\u\cap J_3\u$ be the maximal $\hcx$-invariant subspace contained in $\u$. Then $\h$ is an abelian ideal of $\g$ and there exists a $\hcx$-invariant complementary 
subspace $\q= \operatorname{span}\{e_0,e_1,e_2,e_3\}$ of $\h$ with $e_{\al}=J_{\al}e_0$, such that $e_0\notin \u$, $e_\al \in \u$ and, moreover:
\begin{enumerate}
	\item[$\ri$] $[e_0,e_{\al}]=\mu e_{\al}+v_{\al}$ for some $\mu \in \R$ and $v_{\al}\in \h$, ${\al}=1,2,3$,
	\item[$\rii$] there exists $v_0\in \h$ such that $J_\al v_0=v_\al$ for all $\al$ and moreover,   $v_\al=J_\be v_\ga$ for any  cyclic permutation $(\al , \be , \ga)$  of $(1, 2, 3)$,
	\item[$\riii$] $[e_0,x]=Bx$ for any $x\in\h$, where $B\in\operatorname{End}(\h)$ satisfies $[B,J_\alpha|_\h]=0$, $\alpha=1,2,3$.
\end{enumerate}
\end{theorem}

In other words, $\g$ can be written as $\g=\R e_0\ltimes_A \R^{4n-1}$, where the matrix $A\in\mathfrak{gl}(4n-1,\R)$ defined by the adjoint action of $e_0$ on $\R^{4n-1}$ has the following expression in a basis $\{e_1, e_{2}, e_3\} \cup \mathcal B$ of $\u$,  where $\mathcal B$ is a basis of $\h$:
\begin{equation}\label{matrixL} 
A=\left[
	\begin{array}{ccc|ccc}      
		\mu & & & & &\\
		& \mu & & & 0 & \\
		& & \mu & & & \\
		\hline
		| & | & |& & & \\
		v_1 & v_2 & v_3 & & B &\\
		| & | & | & & &
	\end{array}
	\right], \qquad B\in\mathfrak{gl}(n-1,\mathbb{H})\subset \mathfrak{gl}(4n-4,\R).
\end{equation}

\begin{proof}[Proof of Theorem \ref{characterization}]
We consider  an auxiliary hyperhermitian inner product $\pint$ on $\g$. Let $e_0$ be a non-zero element in $\u^\perp$, so that $\u=e_0^\perp$, and let $e_{\al}=J_{\al}e_0$. This implies that $J_\al e_\be = e_\ga$ for any cyclic permutation $(\al , \be , \ga)$   of $(1, 2, 3)$. Since 
$\q= \operatorname{span}\{e_0,e_1,e_2,e_3\}$ and  $\q^{\perp}$ are $\hcx$-invariant, it follows that $\h=\q^{\perp}$. 

According to Lemma~\ref{hermitian}, for each $\al$, we have that $[e_0,e_{\al}]=\mu_{\al}e_{\al}+v_{\al}$, for some $\mu_{\al}\in \R$ and $v_{\al} \in (\operatorname{span}\{e_0,e_{\al}\})^{\perp}$. 
Let $(\al , \be , \ga)$ be a cyclic permutation of $(1, 2, 3)$. It follows from  \eqref{nijen} that 
\[ 0= N_{J_{\be}}(e_0, e_{\al}) = (\mu_{\al}-\mu_{\ga})e_{\al} +(v_\al-J_\be v_\ga) .   
\]
Since
\[ \la v_\al-J_\be v_\ga , e_\al \ra = \la v_\ga , J_\be e_\al \ra = - \la v_\ga , e_\ga \ra =0,
\]
we obtain that $\mu_{\al}=\mu_{\ga}$ and $v_\al=J_\be v_\ga$. Therefore
\[ \mu_1=\mu_2=\mu_3=:\mu ,
\]
and 
applying $J_\al$ on both sides of $v_\al=J_\be v_\ga$ we obtain $J_\al v_\al=J_\ga v_\ga$. It is easy to see that this is also equal to $J_\be v_\be$. 
Moreover, 
\[ \la v_\al , e_\be \ra =\la J_\be v_\al , J_\be e_\be \ra = \la -v_\ga , - e_0\ra =0,
\]
hence $v_\al \in \h,$ for all $\al$. Setting $-v_0= J_1 v_1=J_2v_2=J_3v_3 \in \h$, 
$\rii$ holds.

In order to show that $\h$ is an ideal, take $x\in \h$. Since $\h \subset \u\cap J_\al \u$  
then from Lemma~\ref{hermitian} it follows that $[e_0, x] \in \u\cap J_\al \u$. As this holds for any $\al$, we have  that $[e_0, x] \in \h$, thus $\h$ is an ideal. 

Finally, to prove $\riii$, let $B= \ad_{e_0}: \h \to \h$. For $x\in \h$ we have that $[e_0, J_\al x]=BJ_\al x$. Again, since $\h \subset \u\cap J_\al \u$, Lemma~\ref{hermitian}  implies that 
$[e_0, J_\al x]=J_\al [e_0 , x]= J_\al Bx$. Therefore, $BJ_\al x =  J_\al Bx$ for any $\al$, and $\riii$ follows.
\end{proof}

\medskip

\begin{remark}
It follows from the proof of Theorem \ref{characterization} that, given any hyperhermitian inner product on the hypercomplex almost abelian Lie algebra $\g$, the $\hcx$-invariant subspace $\mathfrak q$ can be chosen to be orthogonal to the $\hcx$-invariant ideal $\h= \u \cap J_1\u\cap J_2\u\cap J_3\u$.  
\end{remark}

\smallskip

\begin{remark}\label{basis} 
By fixing a basis $\mathcal B$ of $\h$ of the form  $\mathcal B =\{f_j \}\cup \{ J_1f_j \}
\cup \{ J_2f_j \} \cup \{ J_3f_j \}$, for $1\leq j\leq n-1$, the matrix $B$ in \eqref{matrixL} can be expressed as: 
\begin{equation}\label{matrixB}
B= \begin{bmatrix} X & -Y & -Z & -W \\
Y & X & W & -Z \\
Z & -W & X & Y \\
W & Z & -Y & X
\end{bmatrix}, \qquad   X, Y,Z,W \in \mathfrak{gl} (n-1, \R ).
\end{equation}
In this basis, the operators $J_\alpha : \h \to \h$   take the following form:
\begin{equation*}
J_1=\begin{bmatrix}  & -I & & \\
I & & & \\
 & & & -I \\
 & & I & 
\end{bmatrix}, \qquad 
J_2=\begin{bmatrix}  &  & -I &  \\
 & & & I \\
 I & & &  \\
 & -I &  & 
\end{bmatrix}, \qquad
J_3=\begin{bmatrix} & &  & -I   \\
 & &  -I & \\
  & I & &  \\
I  &  &  & 
\end{bmatrix},
\end{equation*}
where $I$ is the $(n-1)\times (n-1)$ identity  matrix.
\end{remark}

\medskip 

We show next that under certain assumptions we may perform a hyperholomorphic change of basis of $\g$ such that the vectors  $v_\al$ in Theorem ~\ref{characterization} vanish for all $\al$. 

\begin{proposition}\label{v=0}
With notation as in Theorem~\ref{characterization}, if $v_\al \in \operatorname{Im}(B-\mu I)$ for some $\al$, then there exists a $\hcx$-invariant complementary 
subspace $\q '= \operatorname{span}\{e_0',e_1',e_2',e_3'\}$ of $\h$ with $e_{\al}'=J_{\al}e_0'$ such that $e_0'\notin \u$, $e_\al '\in\u$, $[e_0',e_\al']=\mu e_\al '$ and $\ad_{e_0'}|_\h =B$. 

In particular, this holds if $\mu$ is not an eigenvalue of $B$.
\end{proposition}

\begin{proof}
 Since $v_\al \in \operatorname{Im}(B-\mu I)$ and this subspace   is $\hcx$-invariant we have that $v_1,v_2,v_3 \in \operatorname{Im}(B-\mu I)$, due to Theorem~\ref{characterization} $\rii$. Let $x_1 \in\h$ such that $v_1= (B-\mu I) x_1$. Setting
 \[
 e_0'=e_0+J_1x_1,\quad e_1'=e_1-x_1,\quad e_2'=e_2-J_3x_1,\quad e_3'=e_3+J_2x_1,
 \]
 it is easily checked that all the conditions in the statement are satisfied. 
\end{proof}

\smallskip 

\begin{example}\label{dim-4}
When the hypercomplex almost abelian Lie algebra $\g$ has dimension $4$, that is,   $\h =\{0\}$,  $\g$ can be written as $\g=\R e_0\ltimes_A \R^3$ with $A=\mu I$ for some $\mu\in \R-\{0\}$. Moreover, according to Lemma \ref{ad-conjugated}, we may assume that $\mu =1$. We will denote this Lie algebra by $\mathfrak s_4$. It appears in the classification of $4$-dimensional hypercomplex Lie algebras given in \cite{Bar}, where it was  proved that $\mathfrak s_4$ admits a unique hypercomplex structure, up to equivalence.\footnote{Two hypercomplex structures $\hcx$ and $\{J_\al'\}$ on a Lie algebra $\g$ are equivalent if there exists a Lie algebra  automorphism $\psi$ of $\g$ such that $\psi J_\al =J_\al'\psi$ for all $\al$.} The simply connected Lie group $S_4$ corresponding to $\mathfrak s_4$ admits a left invariant Riemannian metric $g$ such that $(S_4,g)$ is isometric to the real hyperbolic space $\R H^4$ with its symmetric metric of constant sectional curvature equal to $-1$.
\end{example}

\medskip

\subsection{The Obata connection}  

We compute next the Obata connection of the hypercomplex structure on the almost abelian Lie algebra $\g$ from Theorem~\ref{characterization}. 
\begin{proposition}\label{Obata-flat}
Let $\g$ be an almost abelian Lie algebra with hypercomplex structure $\hcx$ as in Theorem~\ref{characterization}. Then the associated Obata connection is given by:
\[ \nabla _{e_0} e_0= \mu e_0 +v_0, \qquad \nabla_{e_0}u = [e_0, u], \qquad   \nabla_u v=0, \text{ for any }  u \in \u, \; v\in \g .
\]
In particular, $\nabla$ is flat.
\end{proposition}

\begin{proof}
We compute first $\nabla_{e_0}e_0$ using \eqref{obata-alpha}:
\begin{align*}
\nabla_{e_0}e_0 & = \frac12(J_\al[e_\al,e_0]-J_\be[e_0,e_\be]) =\frac12(-J_\al(\mu e_\al+v_\al)-J_\be(\mu e_\be+v_\be))\\
& = \frac12(\mu e_0+v_0+\mu e_0+v_0) = \mu e_0+v_0.
\end{align*}
Now we compute, for $\al=1,2,3$, and $(\al,\be,\ga)$ a cyclic permutation of $(1,2,3)$,
\begin{align*}
\nabla_{e_0}e_\al & = \frac12([e_0,e_\al]-J_\be[e_0,J_\be e_\al]+J_\ga[e_\al,J_\be e_\al]) = \frac12 (\mu e_\al+v_\al-J_\be[e_0,-e_\ga]+J_\ga[e_\al,-e_\ga])\\
& = \frac12 (\mu e_\al+v_\al+J_\be (\mu e_\ga+v_\ga)) = \mu e_\al+v_\al =[e_0,e_\al].
\end{align*}
Since $\nabla_{e_0}e_\al=[e_0,e_\al]$ and the Obata connection $\nabla$ is torsion-free, we have that
$\nabla_{e_\al}e_0=0$.

Next,
\[
\nabla_{e_\al}e_\al  = \frac12(J_\al[-e_0,e_\al]+J_\ga[-e_0,-e_\ga]) = \frac12(\mu e_0+v_0-(\mu e_0+v_0))=0,
\]
and 
\begin{align*}
\nabla_{e_\al}e_\be & = \frac12(J_\al [-e_0,e_\be]-J_\be [e_\al,-e_0])  = \frac12 (-J_\al(\mu e_\be+v_\be)-J_\be(\mu e_\al+v_\al))\\
& = \frac12(-(\mu e_\ga+v_\ga)+(\mu e_\ga+v_\ga)) =0.
\end{align*}
Since $\nabla$ has no torsion, we have that $\nabla_{e_\be}e_\al=0$.

Finally, for $y\in \h$ we calculate
\[
\nabla_{e_0}y  = \frac12( [e_0,y]-J_\be[e_0,J_\be y])= \frac12(By-J_\be B J_\be y) = By = [e_0,y],
\]
since $B$ commutes with $J_\be$. As $\nabla_{e_0}y=[e_0,y]$, we have that $\nabla_y{e_0}=0,\, y\in\h$.

Also, for any $\al$ and $y\in \h$,
\begin{align*}
\nabla_{e_\al}y &= \frac12(J_\al[-e_0,y]+J_\ga[-e_0,J_\be y]) = \frac12(-J_\al By-J_\ga B J_\be y)\\
& = \frac12(-J_\al By-J_\ga J_\be B y) = \frac12(-J_\al By + J_\al By)  = 0.
\end{align*}
It follows that $\nabla_y e_\al=0$ and it is clear that $\nabla_xy=0$ for any $x,y\in\h$. 

The fact that the curvature $R$ of $\nabla$ vanishes is immediate.
\end{proof}

\begin{corollary}
The Obata connection on the simply connected Lie group $G$ associated to the hypercomplex almost abelian Lie algebra $\g$ is geodesically complete if and only if $\mu=0$. 
\end{corollary}

\begin{proof}
Let us define a product on $\g$ as follows: 
\[ \cdot:\g\times\g\to \g, \quad (x,y)\mapsto x\cdot y:=\nabla_x y. \]
Since $\nabla$ is torsion-free and flat, this product is known as a \textit{left-symmetric algebra} structure on $\g$ (LSA for short). The completeness of the connection $\nabla$ can be determined in terms of the LSA structure; indeed, $\nabla$ is geodesically complete if and only if all the right multiplications $\rho(x):\g\to\g$, $\rho(x)y=y\cdot x$, are nilpotent (see \cite{Se}). It is clear from Proposition \ref{Obata-flat} that the right multiplications of the LSA structure  determined by the Obata connection  are nilpotent if and only if $\mu=0$.
\end{proof}

\ 

\section{HKT metrics and the corresponding Bismut connection}\label{section-HKT}

In this section we characterize the hypercomplex almost abelian Lie algebras admitting an HKT metric and then we study the associated Bismut connection. 

\begin{proposition}\label{HKT}
Let $\g$ be a $4n$-dimensional almost abelian Lie algebra equipped with a hyperhermitian structure $(\hcx, \pint)$. Assume that the decomposition $\g=\mathfrak q\oplus \h$ from Theorem \ref{characterization} is orthogonal. Then the metric $\pint$ is HKT if and only if $v_\al=0$ for all $\al$  and $B$ is skew-symmetric.

Moreover, such an HKT metric is hyper-K\"ahler if and only if $\mu=0$. In other words, this HKT metric is hyper-K\"ahler   if and only if $\g$ is unimodular.
\end{proposition}

\begin{proof} Assume that the decomposition $\g=\mathfrak q\oplus \h$ from Theorem \ref{characterization} is orthogonal and that  $\pint$ is HKT. Then \eqref{inv_hkt} holds for all $x,y,z\in \g $, that is, $S_1 (x,y,z)=S_2 (x,y,z)=S_3 (x,y,z)$, where  
\[S_\al (x,y,z):=\la [J_\al x,J_\al y],z\ra +\la [J_\al y,J_\al z],x\ra+\la [J_\al z,J_\al x],y\ra. \]  
Note that   $S_\al (x,y,z)=0$ for all $x,y,z\in \h$, $\al=1,2,3$. Also, $S_\al (e_0,x,y)=0$ for all $x,y\in \h .$ Now let $(\al , \be , \ga )$ be a cyclic permutation of $(1,2,3)$, then 
$S_\al (e_\al , e_\be ,x)=S_\be (e_\al , e_\be ,x)=S_\ga (e_\al , e_\be ,x)$ for any  $x\in \h $. Let us compute:
\begin{eqnarray*}
S_\al (e_\al , e_\be ,x)&=&\la [J_\al e_\al ,J_\al e_\be ], x\ra +\la [J_\al e_\be ,J_\al x], e_\al \ra+\la [J_\al x,J_\al e_\al ], e_\be \ra \\
&=& \la [-e_0,e_\ga  ], x\ra +\la [e_\ga ,J_\al x], -e_0\ra +
\la [J_\al x,-e_0], e_\be\ra=-\la v_\ga , x\ra ,
\end{eqnarray*}
where we have used that $\q$ is orthogonal to $\h$. A similar calculation yields:
\[ S_\be (e_\al , e_\be ,x)= -\la v_\ga , x\ra , \quad x\in \h .
\]
On the other hand, 
\begin{eqnarray*}
S_\ga (e_\al , e_\be ,x)&=&\la [J_\ga e_\al ,J_\ga e_\be ], x\ra +\la [J_\ga e_\be ,J_\ga x], e_\al \ra+\la [J_\ga x,J_\ga e_\al ], e_\be \ra \\
&=& \la [e_\be ,-e_\al  ], x\ra +\la [-e_\al ,J_\al x], e_\al\ra +
\la [J_\al x,e_\be ], e_\be\ra=0 ,
\end{eqnarray*}
since all Lie brackets  vanish. Therefore, $\la v_\ga , x\ra=0$ for all $x\in\h$, which gives $v_\ga =0$, $\ga =1,2,3$. 
In particular, $\q$ is a $4$-dimensional Lie subalgebra of $\g$.  

It is easy to check that $S_\al (e_0 , e_\al ,x)=S_\be (e_0 , e_\al ,x)=S_\ga (e_0 , e_\al ,x)=0$ for all $x\in \h$. We compute next, for $x, y\in\h$:
\begin{eqnarray*}
S_\al (e_\al , x ,y)&=&\la [J_\al e_\al ,J_\al x ], y\ra +\la [J_\al x ,J_\al y], e_\al \ra+\la [J_\al y,J_\al e_\al ], x \ra \\
&=& \la [-e_0,J_\al x ], y\ra +\la [J_\al x ,J_\al y], e_\al \ra +
\la [J_\al y,-e_0], x\ra \\ 
&=& \la -BJ_\al x, y\ra + \la BJ_\al y, x\ra. 
\end{eqnarray*}
We also have:
\[
S_\be (e_\al , x ,y)=\la [J_\be e_\al ,J_\be x ], y\ra +\la [J_\be x ,J_\be y], e_\al \ra+\la [J_\be y,J_\be e_\al ], x \ra =0,
\]
since all Lie brackets vanish. In a similar way, $S_\ga (e_\al , x ,y)=0$. Therefore, \eqref{inv_hkt} implies that $\la -BJ_\al x, y\ra + \la BJ_\al y, x\ra=0$ for all $x,y\in \h$, that is, $BJ_\al +J_\al B^* =0$, where $B^*$ is the adjoint operator. Since $B$ commutes with $J_\al $, this implies $B+B^*=0$, hence $B$ is skew-symmetric. 

Conversely, let $(\hcx, \pint)$ be a hyperhermitian structure on $\g$ and assume that  the decomposition $\g=\mathfrak q\oplus \h$ is orthogonal with  $v_\al =0, \; \al =1,2,3,$ and skew-symmetric $B$. Then the  calculations above  imply that \eqref{inv_hkt} holds for $x\in \h,\, y,z\in \g$. Since any hyperhermitian structure in dimension 4 is HKT (see \cite{GT}), it follows that \eqref{inv_hkt} holds for $x, y,z\in \q$ and $(\hcx, \pint)$ is HKT.

\smallskip

It follows from \cite[Proposition 3.1]{BDF} (see also \cite[Theorem 5.1]{Pa})  that an HKT structure  $(\hcx, \pint)$ is hyper-K\"ahler if and only if $\mu=0$. In other words, this HKT structure is hyper-K\"ahler if and only if $\g$ is unimodular. 
\end{proof}

\begin{remark}
In \cite[Theorem 5.1]{Pa} there is a characterization of almost abelian Lie algebras admitting hyper-K\"ahler and locally conformally hyper-K\"ahler structures. In particular, it follows  that the HKT metrics arising in Proposition~\ref{HKT} with $\mu\neq 0$ are not locally conformally hyper-K\"ahler.
\end{remark}

\medskip

\subsection{The Bismut connection} 
We compute next the Bismut connection of the HKT structure on the almost abelian Lie algebra $\g$ from Proposition \ref{HKT}. 

\begin{proposition}\label{bismut}
Let $\g$ be an almost abelian Lie algebra with an HKT structure $(\hcx, \pint)$ as in Proposition \ref{HKT}. Then, the associated Bismut connection is given by:
\[ \nabla^b_{e_0}|_\q =0,\qquad \nabla^b_{e_0}|_\h =B, \qquad 
\nabla^b_{e_\al}|_\h = 0,
\]
\[ \nabla^b_{e_1}|_\q =\begin{bmatrix} & \mu & & \\
-\mu & & & \\
& & & -\mu \\
& & \mu & 
\end{bmatrix}, \quad 
\nabla^b_{e_2}|_\q =\begin{bmatrix} & & \mu &  \\
 & & &\mu \\
 -\mu & & & \\
 & -\mu & &
\end{bmatrix}, \quad 
\nabla^b_{e_3}|_\q =\begin{bmatrix} & & &\mu   \\
 & & -\mu & \\
 & \mu &  & \\
  -\mu & & &
\end{bmatrix}. 
\]
In particular, the curvature $R^b$ of $\nabla^b$ has the following form:
\[ R^b(e_0, e_\al )=-\mu \nabla^b_{e_\al }, \qquad  R^b(e_\al, e_\be )=2\mu \nabla^b_{e_\ga }, \qquad R^b(e_0, x )=0, \;\; x\in \h . \] 
Moreover, the torsion $3$-form $c$ of $\nabla^b$ is given by $c=2\mu e^1\wedge e^2\wedge e^3$,  and for $\mu\neq 0$, $c$  is non-closed, $\nabla^b c\neq 0$ and the Ricci tensor Ric$^b$ is symmetric.
\end{proposition}

\begin{proof}
We compute $\nabla^b$ using \eqref{Bismut}, so we start by recalling from \cite{Mi} that the Levi-Civita connection $\nabla^g$ in our particular case is given as follows:
\[ \nabla^g_{e_0} e_0=0, \quad \nabla^g_{e_0} u=A_au, \quad \nabla^g_{u} e_0= -A_su, \quad  \nabla^g_{u} v= \la A_su , v\ra e_0 ,\;\; u,v \in \u,
\]
where $A_a$ (resp. $A_s$) is the skew-symmetric (resp. symmetric) part of $A$, that is,
\[ A_a= \begin{bmatrix} 0 & \\
 & B 
\end{bmatrix} \qquad\qquad  A_s= \begin{bmatrix} \mu I & \\
 & 0 
\end{bmatrix}.
\]
Therefore, 
\begin{gather*}
\nonumber    \nabla^g_{e_0} |_\q=0, \quad \nabla^g_{e_0} |_\h=B, \\
\label{L_C}    \nabla^g_{e_\al} e_0= -\mu e_\al , \quad  \nabla^g_{e_\al} e_\al= \mu e_0, \quad  \nabla^g_{e_\al} e_\be=0, \quad \nabla^g_{e_\al} |_\h=0, \\ 
\nonumber \nabla^g_x =0, \; x\in \h  .
\end{gather*}
The torsion $3$-form $c$ can be computed  by $c(x,y,z)= d\omega_\al (J_\al x, J_\al y , J_\al z )$ for any $\al$. Take, for instance, $\al =1$ and it can be checked that $d\omega_1 (J_1 x, J_1 y , J_1 z )=0$  if $x\in \R e_0\oplus \h$, for all $y, z \in \g$. Therefore, if $e^j=\la e_j, \cdot \ra \in \g ^*$, then  $c$ is a multiple of  $e^1\wedge e^2\wedge e^3$,   so we compute    \begin{align*}d\omega_1 (J_1 e_1, J_1 e_2 , J_1 e_3 )&= d\omega_1 (-e_0,e_3, -e_2 )=-d\omega_1 (e_0,e_2,e_3) \\ & =
\omega_1 ([e_0,e_2], e_3)+ \omega_1 ([e_2,e_3], e_0)+\omega_1 ([e_3,e_0], e_2) \\ &= \omega_1 (\mu e_2, e_3)-\omega_1 (\mu e_3, e_2)=2 \mu ,\end{align*}
hence, $c=2\mu e^1\wedge e^2\wedge e^3 .$

Since $c(e_0,x,y)=0$ for all $x,y \in \g$, then $\nabla^b_{e_0}=\nabla^g_{e_0}$. Also, $c(x,y,z)=0$ for all $x\in \h, \; y, z \in \g$, so that $\nabla^b_{x}=\nabla^g_{x}=0$ for all $x\in \h $. Finally, we compute $\nabla^b_{e_\al}$. Again, $c(e_\al,x,y)=0$ for all $x \in \h,\; y \in \g$, then $\nabla^b_{e_\al} |_\h=\nabla^g_{e_\al} |_\h=0$. We compute next
\begin{align*}
    \la \nabla^b_{e_\al} e_0, e_0\ra &= \la \nabla^g_{e_\al} e_0, e_0\ra =-\mu \la e_\alpha , e_0\ra =0, \\
    \la \nabla^b_{e_\al} e_0, e_\al\ra &= \la \nabla^g_{e_\al} e_0, e_\al \ra =-\mu, \\
    \la \nabla^b_{e_\al} e_0, e_\be \ra &=\la \nabla^b_{e_\al} e_0, e_\ga \ra =0,
\end{align*} 
therefore, $\nabla^b_{e_\al} e_0=-\mu e_\al$. Since $J_\al$ is $\nabla^b$-parallel, we obtain:
\begin{align*}
\nabla^b_{e_\al} e_\al &= \nabla^b_{e_\al} J_\al e_0= J_\al \nabla^b_{e_\al} e_0=\mu e_0, \\
\nabla^b_{e_\al} e_\be &= \nabla^b_{e_\al} J_\be e_0= J_\be \nabla^b_{e_\al} e_0=\mu e_\ga.
\end{align*}
 Therefore, $\nabla^b$ is given as in the statement. The computation of $R^b$ follows from the expression of $\nabla^b$ by observing that 
\[ \nabla^b_{[e_0 ,e_\al ]} =\mu \nabla^b_{e_\al} ,  \qquad 
\nabla^b_{e_\al} \nabla^b_{e_\be}=\mu \nabla^b_{e_\ga} =-\nabla^b_{e_\be} \nabla^b _{e_\al} .
\]

It follows from $d e_\al =-\mu e^0 \wedge e^\al$ that
$dc= -6 \mu^2 e^0\wedge e^1\wedge e^2\wedge e^3$,
and using
\[ \nabla ^b _{e_1} e^1=-\mu e^0, \qquad \nabla ^b_{e_2} e^1=-\mu e^3, \qquad \nabla ^b_{e_1} e^3=\mu e^2, 
\]
it turns out that $\nabla^b_{e_1} c=-2\mu ^2 e^0\wedge e^2\wedge e^3$.

Finally, one can easily compute the Ricci tensor Ric$^b$:
\[ \operatorname{Ric}^b(x,y)= \la r^b (x),y \ra , \qquad x, y \in \g,
\] where the Ricci operator $r^b$  is given by:
\[ r^b|_\q =\operatorname{diag} \, (3\mu^2,-5\mu^2,-5\mu^2,-5\mu^2), \qquad \quad r^b|_\h = 0 .
\]
\end{proof}

\section{Eight-dimensional hypercomplex almost abelian Lie algebras and solvmanifolds}
\subsection{Classification of 8-dimensional hypercomplex almost abelian Lie algebras} \label{eight}
In this section we apply Theorem \ref{characterization} to obtain the classification of the $8$-dimensional almost abelian Lie algebras admitting hypercomplex structures (Theorem \ref{classification}). 

\medskip

Let $\g=\R e_0\ltimes_A \R^7$ be an 8-dimensional almost abelian Lie algebra equipped with a hypercomplex structure. It follows from Theorem \ref{characterization} that 
the matrix $A$ in \eqref{matrixL} takes the  form
\begin{equation}\label{matrix-dim8} 
A=\left[
	\begin{array}{ccc|ccc}       
		\mu & & & & &\\
		& \mu & & & 0 & \\
		& & \mu & & & \\
		\hline
		| & | & |& & & \\
		v_1 & v_2 & v_3 & & B &\\
		| & | & | & & &
	\end{array}
	\right], \quad \text{with} \quad v_\alpha\in\R^4,\, B\in \mathfrak{gl} (1,\H).
\end{equation} 
Moreover, following Remark \ref{basis}, there is a basis $\{f_1,f_2:=J_1f_1,f_3:=J_2f_1,f_4:=J_3f_1\}$ such that $B$ is given in this basis by:
\[
B= \begin{bmatrix} \lambda & -y & -z & -w \\
y & \lambda & w & -z \\
z & -w & \lambda & y \\
w & z & -y & \lambda
\end{bmatrix},  \quad \text{ for } \quad \lambda, y,z,w \in \R.
\]
In other words, $B=\lambda I+U$, where $U$ is a skew-symmetric $(4\times 4)$-matrix which commutes with $J_\alpha$ for any $\alpha$. Note that 
\[ \det U=\det(B-\lambda I)=(y^2+z^2+w^2)^2, \]
so that either $U$ is non-singular or $U=0$.
Also, the hypercomplex Lie algebra $\g=\R\ltimes_A\R^7$ is unimodular if and only if $3\mu+4\lambda=0$.

\medskip 

In order to classify these Lie algebras, we consider several cases:

\medskip 

\noindent $(1)$ \textsl{Case $U=0$ (or equivalently, $B=\lambda I$)}: 
\begin{itemize}
    \item If $\mu\neq \lambda$ we may assume $v_\alpha=0$ for all $\alpha$, according to Proposition \ref{v=0}. Then:
    \begin{enumerate}
        \item[$\ri$] If $\mu=0$, $\lambda\neq 0$ we may assume $\lambda=1$ and the only non-vanishing Lie brackets are:
        \[ [e_0,x]=x, \quad x\in\h. \]
        This Lie algebra is not unimodular and it is isomorphic to $\R^3 \times \s_5 $, where $\s_5 =\R e_0\ltimes \h$. The simply connected solvable Lie group $S_5$ with Lie algebra  $\s_5$ admits a left invariant Riemannian metric $g$ such that $(S_5, g)$ is isometric to the real hyperbolic space $ \R H^5$.
        \item[$\rii$] If $\mu\neq 0$ we may assume $\mu =1$, $\lambda \neq 1$, and the non-vanishing Lie brackets are:
        \[ [e_0,e_\al]=e_\al, \quad [e_0,x]=\lambda x, \quad \al=1,2,3, \quad x\in\h. \]
        Note that, according to Lemma \ref{ad-conjugated}, these Lie algebras are not isomorphic for different values of $\lambda$.
        If $\lambda=0$ then $\g=\mathfrak s_4\times \R^4$, where $\mathfrak s_4$ is the $4$-dimensional Lie algebra from Example \ref{dim-4}. Note that  $\g$ is unimodular if and only if $\lambda=-\frac34 $.
    \end{enumerate}
    
    \item If $\mu=\lambda=0$ then $v_\al\neq 0$ for all $\al$ (otherwise $\g$ would be abelian) and therefore $\g$ is $2$-step nilpotent with non-vanishing Lie brackets given by 
    \[ [e_0,e_\alpha]=v_\alpha, \quad \al=1,2,3.\] 
    This Lie algebra  appears in  the second place in the list of examples of Lie algebras with  first Betti number equal to $5$ given in \cite[page 56]{DF1}. Note that  any 8-dimensional hypercomplex nilpotent Lie algebra is either abelian or 2-step,  according to \cite[Theorem~2.1]{DF1}.
    \item If $\mu=\lambda\neq 0$ we may assume $\mu=\lambda=1$ and the only non-vanishing Lie brackets are
    \[ [e_0,e_\al]=e_\al+v_\al, \quad [e_0,x]=x, \quad \text{for} \quad \al=1,2,3 \quad \text{and} \quad  x\in\h.\] 
    In this case $\g$ is not unimodular. There are two isomorphism classes of Lie algebras in this family. When $v_\al =0$ for all $\al$, this Lie algebra is $\s_8$ corresponding to the  real hyperbolic space $\R H^8$. When $v_\al \neq 0$ for all $\al$, we can take $\{ v_0, v_1, v_2, v_3\}$ from Theorem \ref{characterization} as a basis of $\h$ and the Lie bracket is given by:
    \[ [e_0,e_\al]=e_\al+v_\al, \quad [e_0,v_0]=v_0, \quad [e_0,v_\al]=v_\al, \quad \text{for} \quad \al=1,2,3.
    \]
\end{itemize}

\medskip

\noindent $(2)$ \textsl{Case $U\neq 0$ (or equivalently, $\det U\neq 0$}): note that the eigenvalues of $B=\lambda I+U$ are  $\lambda\pm is$ with $s=(y^2+z^2+w^2)^{\frac12}\in \R-\{0\}$. Hence $B-\mu I$ is non-singular and therefore, according to Proposition \ref{v=0}, we may assume $v_\alpha=0$ for all $\alpha$. Moreover, $U$ is diagonalizable over $\C$ and there exists a real basis of $\h$, adapted to the hypercomplex structure, such that the matrix $B$ is given by  
\begin{equation}\label{B-dim8}
B= \begin{bmatrix} \lambda & -s &  &  \\
s & \lambda &  &  \\
 &  & \lambda & s \\
 &  & -s & \lambda
\end{bmatrix},  \quad \text{ for } \quad \lambda \in \R, \; s>0.
\end{equation}

\begin{itemize}
    \item If $\mu=\lambda=0$ then $\g$  is isomorphic to a semidirect product $\g=\R^3\times (\R e_0 \ltimes \R^4 )$, where $e_0$ acts on $\R^4$ by $B$ as in \eqref{B-dim8} with $ s=1, \, \lambda=0$ 
 due to Lemma \ref{ad-conjugated}.  This Lie algebra is unimodular.
    \item If $\mu=0$ and $\lambda\neq 0$, we may assume that  $\lambda=1$. Then $\g$ is isomorphic to $\R^3\times (\R e_0\ltimes_B \R^4)$, with $B$ as in \eqref{B-dim8} with $\lambda=1$. These Lie algebras are non-isomorphic for different values of $s>0$, and they are not unimodular.  
    \item If $\mu\neq 0$ we may assume $\mu=1$ and $\g$  is isomorphic to a semidirect product $\g=\mathfrak{s}_4\ltimes \R^4$, where $\mathfrak{s}_4$ is the $4$-dimensional Lie algebra from Example \ref{dim-4}, and  $e_0$ acts on $\R^4$ by $B$ as in \eqref{B-dim8}. These Lie algebras are non-isomorphic for different complex numbers $\lambda + is$ with $s>0$, and they are unimodular only when $\lambda=-\frac34$.  
\end{itemize} 

\

The paragraphs above can be summarized as follows, where the Lie brackets are given in terms of a basis $\{e_0,e_1,e_2,e_3, f_0, f_1, f_2, f_3\}$.

\begin{theorem}\label{classification} Let $\g$ be an $8$-dimensional almost abelian Lie algebra admitting a hypercomplex structure. Then $\g$ is isomorphic to one and only one of the  following Lie algebras:
\begin{enumerate}
    \item[] $\g_1:$ $[e_0,f_i]=f_i, \quad  0\leq i \leq 3$,
    \item[] $\g_2^\lambda:$ $[e_0,e_\al] = e_\al, \quad [e_0, f_i]=\lambda f_i, \quad \al=1,2,3, \quad 0\leq i \leq 3,\quad \lambda \in \R, $
    \item[] $\g_3:$ $[e_0, e_\al]=f_{\al},\quad \al =1,2,3$,
    \item[] $\g_4:$ $[e_0, e_\al]=e_\al +f_\al , \quad [e_0, f_i]=f_i, \quad \al =1,2,3,\quad 0\leq i \leq 3$,
    \item[] $\g_5:$ $[e_0, f_0]=f_1, \quad [e_0, f_1]=-f_0, \quad [e_0, f_2]=-f_3,\quad [e_0, f_3]=f_2$,
    \item[] $\g_6^s:$ $[e_0, f_0]=f_0+sf_1, \quad [e_0, f_1]=-sf_0+f_1,$ \\ \hspace*{.6cm} $[e_0, f_2]=f_2-sf_3,\quad [e_0, f_3]=sf_2+f_3, \quad\! s>0$,
    \item[] $\g_7^{\lambda,s}:$ $[e_0, e_\al]=e_\al, \quad \al=1,2,3,\quad [e_0, f_0]= \lambda  f_0+sf_1, \quad [e_0, f_1]=-sf_0+\lambda f_1,$ \\ \hspace*{.85cm} $[e_0, f_2]=\lambda f_2-sf_3,\quad [e_0, f_3]=sf_2+ \lambda f_3, \quad \lambda \in \R,\; s>0$.
    \end{enumerate}
    Moreover, the only unimodular Lie algebras are $\g_2^{-\frac34}, \; \g_3, \; \g_ 5$ and $\g_7^{{-\frac34},s}$.
\end{theorem}

\medskip

\subsection{Existence of lattices and associated solvmanifolds}\label{section-lattices}

In this section we will prove (Corollary \ref{lattices-dim8}) that only two of the simply connected Lie groups corresponding to the Lie algebras in Theorem \ref{classification} admit lattices.  We then exhibit a countable family of pairwise non-homeomorphic hypercomplex nilmanifolds (Proposition \ref{G3}) and we also show that any 8-dimensional compact flat hyper-K\"ahler manifold can be realized as an almost abelian solvmanifold with an invariant hyper-K\"ahler structure (Theorem \ref{8dim-HK}). 

\smallskip

We begin with the following non-existence result. 

\begin{proposition}\label{lattices}
Let $\g$ be an almost abelian Lie algebra   isomorphic to $\g_2^{-\frac34}$ or $\g_7^{{-\frac34},s}$ for some $s>0$. 
Then its associated simply connected Lie group $G$ does not admit lattices.
\end{proposition}

\begin{proof}
The Lie algebra $\g $ as in the statement is a semidirect product  $\g=\R e_0 \ltimes _A \R^7$, with $A=D+S$, where $D$ and $S$ are   given by:
\begin{equation*} D=  \operatorname{diag}\, \left(1,1,1,-\frac34,-\frac34,-\frac34,-\frac34\right), \qquad \quad
S=\left[	\begin{array}{c|c}  
0_3 & \\
\hline 
  & U \end{array}\right], 
\end{equation*}
for a skew-symmetric $(4\times 4)$-matrix $U$, where  $0_3$ is  the $3\times 3$ zero matrix. 

The simply connected Lie group associated to $\g$ can be written as $G=\R\ltimes_\varphi \R^7$, where $\varphi:\R\to \operatorname{GL}(7,\R)$ is the homomorphism given by $\varphi(t)=\exp(tA)$. According to 
Proposition \ref{latt}, the Lie group $G$ admits lattices if and only if there exists $t_0\neq 0$ such that $\varphi(t_0)$ is conjugate to an integer matrix. 

Let us assume that such $t_0\neq 0$ exists. Note that
\begin{equation}\label{phi-t0}
\varphi(t_0)=\exp(t_0A)=\left[
	\begin{array}{ccc|ccc}       
		e^{t_0} & & & & &\\
		& e^{t_0} & & & & \\
		& & e^{t_0} & & & \\
		\hline
		& & & & & \\
	    & & & & e^{-\frac34 t_0}\exp(t_0U) &\\
		& & & & &
	\end{array}
	\right].
\end{equation}
Therefore the characteristic polynomial $p(x)$ of $\varphi(t_0)$ has integer coefficients:
\[ p(x)=x^7-m_6 x^6+m_5x^5+\cdots+m_1x-1\in\Z[x].\]
Note that $p(x)$ has a triple root $e^{t_0}$, and the other roots are complex numbers with the same length, namely $e^{-\frac34 t_0}$ (since $\exp(t_0U)$ is orthogonal and its eigenvalues have length 1). 

\medskip 

\textsl{Claim:} $p(x)$ is irreducible over $\Z[x]$. Indeed, assume that $p=qr$ with $q,r\in\Z[x]$, $\deg q>0$. Then $q(0)r(0)=p(0)=-1$, so that $q(0)=\pm 1$, $r(0)=\mp 1$. Since $q(0)$ is the product up to sign of its roots, which form a subset of the roots of $p$, we have that
\[ 1=|q(0)|=(e^{t_0})^k(e^{-\frac34 t_0})^j =(e^{t_0})^{k-\frac34 j}, \quad \text{with} \quad 0\leq k\leq 3, \, 0\leq j\leq 4, \, k+j>0.\]
Since $t_0\neq 0$, this implies that $4k=3j$, in particular $k\equiv 0 \, (\operatorname{mod} 3)$, and therefore the only possible option is $k=3, j=4$. Thus, $q=p,\, r=1$ and $p$ is irreducible.

\medskip 

Now, let $m(x)$ denote the minimal polynomial of $\varphi(t_0)$. It is known that $m(x)\in \Z[x]$ and, moreover, $p(x)=m(x)h(x)$ with $h(x)\in\Z[x]$ (see, for instance,   \cite[Appendix B]{Bo}). Since $p(x)$ is irreducible over $\Z[x]$ we have that $p(x)=m(x)$. However, it follows from \eqref{phi-t0} that the factor $(x-e^{t_0})$ appears three times in the factorization of $p(x)$, whereas it only appears once in the factorization of $m(x)$. This is a contradiction and therefore we must have $t_0=0$. As a consequence, $G$ does not admit lattices.
\end{proof}

\smallskip

\begin{corollary}\label{lattices-dim8}
Let $G$ be a simply connected  8-dimensional almost abelian Lie group equipped with a left invariant hypercomplex structure. If $G$ admits lattices then its Lie algebra $\g$ is isomorphic to either $\g_3$ or $\g_5$ from Theorem \ref{classification}.  
\end{corollary}

\begin{proof}
The only unimodular Lie algebras in Theorem \ref{classification} are $\g_2^{-\frac34},\g_3, \g_5$ and $\g_7^{-\frac34,s}$. It follows from Proposition \ref{lattices} that a simply connected Lie group with Lie algebra  $\g_2^{-\frac34}$ or  $\g_7^{-\frac34,s}$ does not admit lattices.

For the nilpotent Lie algebra $\g_3$, since its structure constants are rational numbers we have that the associated simply connected Lie group $G_3$ admits lattices, according to Theorem \ref{Mal}. 

For $\g_5$, the corresponding matrix $A$ is given by\footnote{For  square matrices $X,Y$, we denote  $X\oplus Y=\begin{bmatrix}
    X&0 \\ 0&Y
\end{bmatrix}$. We use a similar notation for 3 or more matrices.}
\begin{equation}\label{A-HK} A=\left[0_3\right] \oplus \begin{bmatrix}0&-1 \\1&0 \end{bmatrix} \oplus 
\begin{bmatrix} 0& 1\\ -1&0 \end{bmatrix}. 
\end{equation}
Let us consider $t_m=\frac{2\pi}{m}$ for $m\in \{1,2,3,4,6\}$. It is easy to verify that the matrix $\exp(t_mA)$ is either integer or conjugate to an integer matrix (see the discussion preceding Theorem \ref{8dim-HK} below). Therefore, according to Proposition \ref{latt}, the simply connected Lie group $G_5$ associated to $\g_5$ admits lattices.
\end{proof}

\medskip

In the remainder of this section we carry out a detailed analysis of the lattices in the Lie groups $G_3$ and $G_5$  and we study  topological aspects of the corresponding nil- and solvmanifolds.  

\begin{proposition}\label{G3}\
\begin{enumerate}
    \item[$\ri$] For any lattice $\Gamma$ in $G_3$, the Betti numbers $b_j$ of $\Gamma\backslash G_3$ are given by
    \[ b_0=b_8=1, \quad b_1=b_7=5, \quad b_2=b_6=16, \quad b_3=b_5=30, \quad b_4= 36. \]
    \item[$\rii$] For any $k\in \N$ there is a lattice $\Gamma_k$ in $G_3$ such that  the nilmanifolds $\Gamma_k\backslash G_3$ are pairwise non-homeomorphic.
\end{enumerate} 
\end{proposition}

\begin{proof}
For $\ri$, the Betti numbers $b_j, \, 0\leq j\leq 4$, can be computed using the Chevalley-Eilenberg complex of $\g_3$, due to \eqref{deRham}. The other Betti numbers can be deduced by Poincaré duality.

For $\rii$, the almost abelian Lie algebra $\g_3=\R e_0\ltimes_A \R^7$ is determined by the matrix 
\[ A=\left[
	\begin{array}{ccc|cccc}       
		& & & & & &\\
		& 0_3 & & & & &\\
		& &  & & & &\\
		\hline
		0 & 0 & 0 & & & &\\
	    1 & 0 & 0 & & & 0_4 &\\
		0 & 1 & 0 & & & &\\
		0 & 0 & 1 & & & &
	\end{array}
	\right]. \]
Therefore, $G_3=\R\ltimes_{\varphi} \R^7$ with 
\[ \varphi(t)=\exp(tA)=\left[
	\begin{array}{ccc|cccc}       
		1& & & & & &\\
		& 1 & & & & &\\
		& & 1 & & & &\\
		\hline
		0 & 0 & 0 & 1 & & &\\
	    t & 0 & 0 & & 1 & &\\
		0 & t & 0 & & & 1 &\\
		0 & 0 & t & & & & 1
	\end{array}
	\right].\]
For each $k\in\N$, let us denote $E_k:=\varphi(k)$; note that $E_k\in \operatorname{SL}(7,\Z)$. Thus, $\Gamma_k:=k\Z\ltimes_{E_k} \Z^7$ is a lattice in $G_3$, which is isomorphic to $\Gamma_k\cong \Z\ltimes_{E_k}\Z^7$ in a natural way. It follows from \cite[Proposition 4.10]{Tol2} that $[\Gamma_k,\Gamma_k]=0\Z\oplus \operatorname{Im}(I-E_k)=\{(0,0,0,0,0,kp,kq,kr)\mid p,q,r\in \Z \}$. Therefore we obtain
\begin{equation}\label{abel}
\Gamma_k/[\Gamma_k,\Gamma_k]\cong \Z^5\oplus (\Z_k)^3, \end{equation} 
which implies that the lattices are pairwise non-isomorphic. Since $\Gamma_k=\pi_1(\Gamma_k\backslash G_3)$, the proof is   complete.
\end{proof}

\begin{remark}
Due to Hurewicz theorem and \eqref{abel}, we have that $H_1(\Gamma_k\backslash G_3,\Z)\cong \Z^5\oplus (\Z_k)^3$.
\end{remark}

\medskip

We consider next solvmanifolds associated to the Lie group $G_5$ from the proof of Corollary \ref{lattices-dim8}. Since the matrix $A$ in \eqref{A-HK} which gives rise to $\g_5$ is skew-symmetric with $\mu=0$, the   following result is a consequence of Proposition \ref{HKT} (see also \cite[\S3.1]{BDF}).

\begin{corollary}
Any hypercomplex solvmanifold $\Gamma\backslash G_5$ admits a flat hyper-K\"ahler metric. 
 
\end{corollary}

\begin{remark}
If $\Gamma$ is a lattice in $G_5$ then $\Gamma$ is the fundamental group of a compact flat manifold, so it is isomorphic to a  Bieberbach group, i.e. a co-compact torsion-free discrete subgroup of the isometry group of $\R^8$ with the standard flat metric \cite{Bieb1,Bieb2, Ch}. 
\end{remark}

\medskip

The classification of $8$-dimensional compact flat hyper-K\"ahler manifolds was carried out by Whitt in \cite{Whitt}, where exactly twelve of them were determined, up to affine diffeomorphism. They are 
 distinguished by their Riemannian holonomy group and the abelianization of their fundamental group, that is, their first homology group. We list them in Table \ref{tabla}.

 \smallskip
 \begin{table}[ht] 
\begin{tabular}{|c|c|c||c|c|c|} \hline
 $M$ & $\operatorname{Hol}(M,g)$ & $H_1(M,\Z)$ & $M$ & $\operatorname{Hol}(M,g)$ & $H_1(M,\Z)$
 \\ \hline
 $M_1$ & $\{e\}$ &  $\Z^8$ & $M_{3,1}$ & $\Z_3$ & $\Z^4 \oplus \Z_3$\\
 $M_{2,0}$ & $\Z_2$ & $\Z^4 \oplus (\Z_2)^4$ & $M_{3,2}$ & $\Z_3$ & $\Z^4 $ \\
 $M_{2,1}$ & $\Z_2$ & $\Z^4 \oplus (\Z_2)^3$ & $M_{4,0}$ & $\Z_4$ & $\Z^4 \oplus (\Z_2)^2$ \\
 $M_{2,2}$ & $\Z_2$ & $\Z^4 \oplus (\Z_2)^2$ & $M_{4,1}$ & $\Z_4$ & $\Z^4 \oplus \Z_2$ \\
 $M_{2,3}$ & $\Z_2$ & $\Z^4 \oplus \Z_2$ & $M_{4,2}$ & $\Z_4$ & $\Z^4 $ \\
$M_{3,0}$ & $\Z_3$ & $\Z^4 \oplus (\Z_3)^2$ & $M_6$ & $\Z_6$ & $\Z^4$ \\ 
 \hline
\end{tabular} 
\caption{\label{tabla} $8$-dimensional compact flat hyper-K\"ahler manifolds.}
\end{table}

\smallskip

Next we analyze  in detail the lattices in $G_5$ and establish a correspondence between the associated solvmanifolds with those obtained by Whitt. 

\medskip

Recall the notation in the proof of Corollary \ref{lattices-dim8}. We note first that if $\Gamma_m$ denotes a lattice in $G_5$ determined by $t_m=\frac{2\pi}{m}$, for $m\in \{1,2,3,4,6\}$, then the holonomy group of $\Gamma_m\backslash G_5$ is $\Z_m$ for $m\neq 1$ and $\{e\}$ for $m=1$, according to  \cite[Theorem 3.7]{Tol1}. 

We will denote by $\varphi:\R\to \operatorname{GL}(7,\R)$ the group homomorphism given by  $\varphi(t)=\exp(t A)$, with $A$ as in \eqref{A-HK}. 

\smallskip 

When $m=1$ we have $\varphi(t_1)=\varphi(2\pi)=I$, so that the lattice $\Gamma_{1}=2\pi\Z \ltimes_{\varphi(2\pi)} \Z^7$ is abelian. Therefore, according to Theorem \ref{solv-isom}, $\Gamma_{1}\backslash G_5$ is the $8$-dimensional flat torus, corresponding to the manifold $M_1$ in Table \ref{tabla}.

\smallskip 

When $m=2$, so that $t_2=\pi$, the associated solvmanifolds have holonomy $\Z_2$. Let us denote 
\[ E_{2,0}:=\exp(\pi A)= \operatorname{diag}(1,1,1,-1,-1,-1,-1)\in \operatorname{SL}(7,\Z).\]
Therefore a lattice in $G_5$ is given by $\Gamma_{2,0}=\pi \Z\ltimes_{E_{2,0}}\Z^7\cong \Z\ltimes_{E_{2,0}}\Z^7$. Consider next the following matrices in $\operatorname{SL}(7,\Z)$:
\begin{gather*}
E_{2,1}=\begin{bmatrix}
1 & 0 \\ 1 & -1 \end{bmatrix}\oplus [-I_3]\oplus [I_2], \quad E_{2,2}=\begin{bmatrix}
1 & 0 \\ 1 & -1 \end{bmatrix}\oplus\begin{bmatrix}
1 & 0 \\ 1 & -1 \end{bmatrix}\oplus [-I_2]\oplus [1], \\
E_{2,3}=\begin{bmatrix}
1 & 0 \\ 1 & -1 \end{bmatrix}\oplus\begin{bmatrix}
1 & 0 \\ 1 & -1 \end{bmatrix}\oplus\begin{bmatrix}
1 & 0 \\ 1 & -1 \end{bmatrix}\oplus [-1].
\end{gather*}
All of them are conjugate to $E_{2,0}$ in $\operatorname{GL}(7,\R)$ since they have the same Jordan form, which is exactly $E_{2,0}$. Therefore there are matrices $P_{2,j}\in \operatorname{GL}(7,\R)$ such that $\Gamma_{2,j}:=\pi \Z\ltimes_{\varphi(\pi)} P_{2,j}\Z^7\cong \Z\ltimes_{E_{2,j}}\Z^7$ are lattices in $G_5$, for $j=1,2,3$. Moreover, these lattices are pairwise non-isomorphic for $j=0,\ldots,3$, since it follows from \cite[Proposition 4.10]{Tol2} that $[\Gamma_{2,j},\Gamma_{2,j}]=0\Z\oplus \operatorname{Im}(I_7-E_{2,j})$ and using this it is easy to establish that
\[ \Gamma_{2,j}/[\Gamma_{2,j},\Gamma_{2,j}]\cong \Z^4\oplus (\Z_2)^{4-j}, \quad 0\leq j \leq 3.  \]
The four solvmanifolds $\Gamma_{2,j}\backslash G_5$ correspond to $M_{2,j}$  in Table \ref{tabla} for $0\leq j\leq 3$. 

\smallskip

When $m=4$, so that $t_4=\frac{\pi}{2}$, the associated solvmanifolds have holonomy $\Z_4$. Let us denote
\[ E_{4,0}:=\exp\left(\frac{\pi}{2} A\right)= \left[I_3\right] \oplus \begin{bmatrix}0&-1 \\1&0 \end{bmatrix} \oplus 
\begin{bmatrix} 0& 1\\ -1&0 \end{bmatrix}\in \operatorname{SL}(7,\Z).\]
Therefore a lattice in $G_5$ is given by $\Gamma_{4,0}=\pi \Z\ltimes_{E_{4,0}}\Z^7\cong \Z\ltimes_{E_{4,0}}\Z^7$. Consider next the following matrices in $\operatorname{SL}(7,\Z)$:
\[
E_{4,1}=[I_2]\oplus\begin{bmatrix} 0& 1\\ -1&0 \end{bmatrix}
\oplus \begin{bmatrix}
1 & 0 & 0  \\ -1 & 0 & 1 \\ 0 & -1 & 0 \end{bmatrix}, \quad E_{4,2}=[1]\oplus \begin{bmatrix} 0 & 0 & 0 & 0 & -1 & 0 \\ 0 & 0 & 0 & 0 & 0 & -1 \\ 0 & 1 & 0 & 0 & 0 & 1 \\ -1 & 0 & 0 & 0 & -1 & 0 \\ 0 & 0 & 0 & 1 & 1 & 0 \\ 0 & 0 & -1 & 0 & 0 & 1 \\ \end{bmatrix}.
\]
The matrices $E_{4,0}, E_{4,1}$ and $E_{4,2}$ are pairwise conjugate in $\operatorname{GL}(7,\R)$ since they have the same Jordan form. Therefore there are matrices $P_{4,j}\in \operatorname{GL}(7,\R)$ such that $\Gamma_{4,j}:=\frac{\pi}{2} \Z\ltimes_{\varphi(\frac{\pi}{2})} P_{4,j}\Z^7\cong \Z\ltimes_{E_{4,j}}\Z^7$ are lattices in $G_5$, for $j=1,2$. Moreover, these lattices are pairwise non-isomorphic for $j=0,1,2$, since it is easily verified that 
\[ \Gamma_{4,j}/[\Gamma_{4,j},\Gamma_{4,j}]\cong \Z^4\oplus (\Z_2)^{2-j}, \quad 0\leq j \leq 2.  \]
The three solvmanifolds $\Gamma_{4,j}\backslash G_5$ correspond to $M_{4,j}$  in Table \ref{tabla} for $0\leq j\leq 2$.

\smallskip

When $m=3$, so that $t_3=\frac{2\pi}{3}$, the associated solvmanifolds have holonomy $\Z_3$. We have that
\[ \varphi\left(\frac{2\pi}{3} \right)= \left[I_3\right] \oplus \begin{bmatrix} -\frac12 & -\frac{\sqrt 3}{2} \\ \frac{\sqrt 3}{2} & -\frac12 \end{bmatrix} \oplus \begin{bmatrix} -\frac12 & \frac{\sqrt 3}{2} \\ -\frac{\sqrt 3}{2} & -\frac12 \end{bmatrix}, \]
and it is easy to verify that this matrix is conjugate to
\[ E_{3,0}:=[I_3]\oplus \begin{bmatrix} 0 & -1 \\ 1 & -1 \end{bmatrix} \oplus \begin{bmatrix} 0 & 1 \\ -1 & -1 \end{bmatrix}\in \operatorname{SL}(7,\Z), \] 
Therefore a lattice in $G_5$ is given by $\Gamma_{3,0}=\frac{2\pi}{3} \Z\ltimes_{\varphi(\frac{2\pi}{3})}P_{3,0}\Z^7\cong \Z\ltimes_{E_{3,0}}\Z^7$ for some  $P_{3,0}\in \operatorname{GL}(7,\R)$. Consider next the following matrices in $\operatorname{SL}(7,\Z)$:
\[
E_{3,1}=[I_2]\oplus\begin{bmatrix} -1 & -1 \\ 1 & 0 \end{bmatrix} \oplus \begin{bmatrix}
0 & 0 & 1  \\ -1 & 0 & 0 \\ 0 & -1 & 0 \end{bmatrix}, \quad E_{3,2}=[1]\oplus \begin{bmatrix} 0 & -1 & 0 \\ 0 & 0 & 1 \\ -1 & 0 & 0 \end{bmatrix}\oplus\begin{bmatrix}
0 & 1 & 0 \\ 0 & 0 & -1 \\ -1 & 0 & 0 \end{bmatrix}.
\]
The matrices $E_{3,0}, E_{3,1}$ and $E_{3,2}$ are pairwise conjugate in $\operatorname{GL}(7,\R)$ since they have the same Jordan form. Therefore there are matrices $P_{3,j}\in \operatorname{GL}(7,\R)$ such that $\Gamma_{3,j}:=\frac{2\pi}{3} \Z\ltimes_{\varphi(\frac{2\pi}{3})} P_{3,j}\Z^7\cong \Z\ltimes_{E_{3,j}}\Z^7$ are lattices in $G_5$, for $j=1,2$. Moreover, these lattices are pairwise non-isomorphic for $j=0,1,2$, since it is easily verified that 
\[ \Gamma_{3,j}/[\Gamma_{3,j},\Gamma_{3,j}]\cong \Z^4\oplus (\Z_3)^{2-j}, \quad 0\leq j \leq 2.  \]
The three solvmanifolds $\Gamma_{3,j}\backslash G_5$ correspond to $M_{3,j}$  in Table \ref{tabla} for $0\leq j\leq 2$.

Finally, for $m=6$ we have $t_6=\frac{\pi}{3}$ and the associated solvmanifolds have holonomy $\Z_6$. We have that
\[ \varphi\left(\frac{\pi}{3} \right)= \left[I_3\right] \oplus \begin{bmatrix} \frac12 & -\frac{\sqrt 3}{2} \\ \frac{\sqrt 3}{2} & \frac12 \end{bmatrix} \oplus \begin{bmatrix} \frac12 & \frac{\sqrt 3}{2} \\ -\frac{\sqrt 3}{2} & \frac12 \end{bmatrix}, \]
and it is easy to verify that this matrix is conjugate to
\[ E_{6}:=[I_3]\oplus \begin{bmatrix} 0 & -1 \\ 1 & 1 \end{bmatrix} \oplus \begin{bmatrix} 0 & 1 \\ -1 & 1 \end{bmatrix}\in \operatorname{SL}(7,\Z). \] 
Therefore a lattice in $G_5$ is given by $\Gamma_{6}=\frac{\pi}{3} \Z\ltimes_{\varphi\left(\frac{\pi}{3} \right)}P_6\Z^7\cong \Z\ltimes_{E_{6}}\Z^7$ for some  $P_{6}\in \operatorname{GL}(7,\R)$. It is readily verified that 
\[ \Gamma_{6}/[\Gamma_{6},\Gamma_{6}]\cong \Z^4.  \]
The associated solvmanifold $\Gamma_{6}\backslash G_5$ corresponds to  $M_6$ in Table \ref{tabla}.

\

Summarizing we have:

\begin{theorem}\label{8dim-HK}
Any $8$-dimensional compact flat hyper-K\"ahler manifold is a solvmanifold $\Gamma\backslash G_5$ equipped with an invariant hyper-K\"ahler structure.
\end{theorem}

\medskip

\begin{remark}
The Betti numbers of the non-toral manifolds in Theorem \ref{8dim-HK} have been computed in \cite{DM}: for a manifold with Riemannian holonomy $\Z_2$ the Betti numbers are 
\[ b_0=b_8=1, \quad b_1=b_7=4, \quad b_2=b_6=12, \quad b_3=b_5=28, \quad b_4=38, \]
while for a manifold with holonomy $\Z_3, \Z_4$ or $\Z_6$ the Betti numbers are
\[ b_0=b_8=1, \quad b_1=b_7=4, \quad b_2=b_6=10, \quad b_3=b_5=20, \quad b_4=26. \]
\end{remark}

\begin{remark}
In dimension 12 there are compact flat hyper-K\"ahler manifolds which are not (homeomorphic to) a solvmanifold. Indeed, in \cite{DM} there is an example of a 12-dimensional compact flat hyper-K\"ahler manifold with vanishing first Betti number. This manifold is not a solvmanifold due to \eqref{betti1}. 
\end{remark}

\ 

\section{Hypercomplex almost abelian solvmanifolds in higher dimensions}\label{higher}

In this section we start by  providing  examples of hypercomplex almost abelian nilmanifolds and solvmanifolds arising from Theorem \ref{characterization}. Then, given a hypercomplex almost abelian Lie group, we use its flat Obata connection in order to construct another such Lie group with twice the dimension. Iterating this process we obtain almost abelian Lie groups endowed with Clifford structures.  We also study  the tangent  bundle of any hypercomplex almost abelian Lie group.

\medskip

\begin{example}\label{nilp}
We build nilpotent almost abelian Lie algebras $\g=\R e_0\ltimes_A \R^{4n-1}$, $n\geq 2$, equipped with a hypercomplex structure. In order to do so, we note that the matrix $A$ in \eqref{matrixL} is nilpotent when $\mu=0$ and the matrix $B$ in \eqref{matrixB} satisfies: $X$ is a strictly lower triangular matrix and $Y=Z=W=0$. Indeed, $A$ is a strictly lower triangular matrix.

If $X=0$ and $v_\al\neq 0$ for all $\alpha$ then $\g\cong \g_3\times \R^{4n-8}$, where $\g_3$ is the nilpotent 8-dimensional Lie algebra appearing in Theorem \ref{classification}. Moreover, this decomposition is preserved by the hypercomplex structure. Taking direct products of  lattices $\Gamma$ in $G_3$ with lattices in $\R^{4n-8}$ we obtain lattices in the simply connected Lie group $G$ associated to $\g$ such that the corresponding nilmanifolds are hypercomplex products $(\Gamma\backslash G_3)\times \mathbb{T}^{4n-8}$, where $\mathbb{T}^{4n-8}$ is a torus of real  dimension $4n-8$. 

If $X\neq 0$ then $\dim \g\geq 12$. Let $k= \min \{j\in\N :X^j=0\}$. If $v_\al=0$ for all $\al$ then   $\g$ is $k$-step nilpotent. If $v_\alpha\neq 0$ for all $\al$, we can choose a basis $\mathcal B$ of $\h$ adapted to the hypercomplex structure, $\mathcal B=\{f_1,\ldots, f_{4n-4}\}$, such that $f_1=v_0, f_n=v_1, f_{2n-1}=v_2,f_{3n-2}=v_3$, with $v_0,\ldots, v_3$ as in Theorem \ref{characterization}. Therefore the [$(4n-4)\times 3$]-block of $A$ in the left bottom corner has entries equal either to $0$ or $1$, and it can be shown that $\g$ is either $k$-step or $(k+1)$-step nilpotent. In both cases,  choosing $X$ with all its entries in $\Q$ (or equivalently in $\Z$, according to Lemma \ref{ad-conjugated}), it follows from Theorem \ref{Mal} that the corresponding simply connected Lie group admits lattices. In this way we produce many examples of hypercomplex nilmanifolds which are not hyper-K\"ahler since non-toral nilmanifolds do not admit K\"ahler metrics (\cite{BG}).
\end{example}

\smallskip

\begin{remark}
The 12-dimensional hypercomplex nilpotent Lie algebra appearing in \cite[Section 11.3]{LW} belongs to the family exhibited in Example \ref{nilp}, with $X=\begin{bmatrix}
    0&0\\ 1&0
\end{bmatrix}$ and $v_\al=0$ for all $\al$.
\end{remark}

\begin{remark}
The holonomy group of the Obata connection $\nabla$ on the hypercomplex nilmanifolds from Example \ref{nilp} is contained in $\operatorname{SL}(n, \H)$ (\cite{BDV}). Since $\nabla$ is flat (see Proposition \ref{Obata-flat}),  $\operatorname{Hol}(\nabla)$  is  discrete. 
\end{remark}

\medskip

\begin{example}\label{hcx-non-HK}
We exhibit  solvmanifolds which are not nilmanifolds and do not admit hyper-K\"ahler metrics. 
For any $n\geq 1$, consider the almost abelian Lie algebra $\g_n=\R\ltimes_A \R^{8n+3}$ where $A$ is the diagonal matrix 
\[ A=\operatorname{diag}(0,0,0,1,-1,1,-1,\ldots,1,-1)\in \mathfrak{gl}(8n+3,\R).\]
Then $A$ is given as in \eqref{matrixL} with $\mu=0$, $v_\alpha=0$ and $B$ as in \eqref{matrixB}, where $Y=Z=W=0$
and 
\[ X=\operatorname{diag}(1,-1,1,-1,\ldots,1,-1)\in \mathfrak{gl}(2n,\R).\] 
Therefore $\g_n$ carries a hypercomplex structure. 

For $m\in\N$, $m\geq 3$, let $t_m=\log \frac{m+\sqrt{m^2-4}}{2}$. Then it is easy to verify that $\exp\left(t_m\begin{bmatrix}
1 & 0 \\ 0 & -1\end{bmatrix}\right)$ is conjugate to $\begin{bmatrix} 0 & -1 \\ 1 & m \end{bmatrix}\in \operatorname{SL}(2,\Z)$ and clearly $\exp(t_m A)$ is conjugate to 
\[ E_m:=I_3\oplus \begin{bmatrix} 0 & -1 \\ 1 & m \end{bmatrix}\oplus \cdots \oplus \begin{bmatrix} 0 & -1 \\ 1 & m \end{bmatrix},\] 
where the $(2\times 2)$-block appears $4n$ times. 
Hence, according to Proposition \ref{latt}, the corresponding simply connected almost abelian Lie group $G_n$ admits lattices $\Gamma^n_m=t_m\Z \ltimes P\Z^{8n+3}$, for $m\in\N$, $m\geq 3$, for some $P\in \operatorname{GL}(8n+3,\R)$ with $P^{-1}\exp(t_m A)P=E_m$. The solvmanifolds $M^n_m:=\Gamma^n_m\backslash G_n$ inherit a hypercomplex structure. Moreover, it can be seen that
\[ \Gamma^n_m/[\Gamma^n_m,\Gamma^n_m]\cong \Z^4\oplus (\Z_{m-2})^{4n}, \]
therefore these solvmanifolds are pairwise non-homeomorphic for any choice of $n$ and $m$. 

Note that the Lie algebra $\g_n$ is completely solvable; this fact has some important consequences. First, none of the solvmanifolds $M^n_m$ is homeomorphic to a nilmanifold, due to Theorem \ref{Saito}. Second, $M^n_m$ does not admit any K\"ahler metric, due to the Benson-Gordon conjecture, proved in \cite{Hase} (see also \cite{BC}). In particular, these hypercomplex solvmanifolds do 
not admit any hyper-K\"ahler metric.

Since $\g_n$ is completely solvable, it follows from \eqref{deRham} that the Betti numbers $b_k$ of $M^n_m$ are given by $b_k=\dim H^k(\g_n)$, and they do not depend on $m$. It can be seen, using the corresponding  Chevalley-Eilenberg complex, that
\begin{eqnarray}\label{betti-hcx} 
b_{2k} & = &\binom{4n}{k}^2+6\binom{4n}{k-1}^2+\binom{4n}{k-2}^2, \qquad 0\leq k\leq 4n+2,
\\ \nonumber b_{2k+1} &=& 4\left[\binom{4n}{k}^2+\binom{4n}{k-1}^2\right], \qquad \qquad 0\leq k\leq 4n+1, 
\end{eqnarray}
where $\binom{p}{q}=0$ if $p<q$ or $q<0$.
\end{example}

\smallskip

\begin{remark}
It is known that on any $4p$-dimensional compact  hyper-K\"ahler manifold the odd-degree Betti numbers $b_{2k+1}$ are divisible by $4$ (\cite{Wak}) and the following relation holds (\cite{Sal}):
\begin{equation}\label{salamon}
 p\, b_{2p}=2 \sum_{j=1}^{2p} (-1)^j (3j^2-p)b_{2p-j}.
\end{equation}
We note that even though the solvmanifolds $M^n_m$ from Example \ref{hcx-non-HK} do not admit any hyper-K\"ahler metric, their Betti numbers satisfy the same relations. Indeed, it follows from \eqref{betti-hcx} that $b_{2k+1}$ is divisible by $4$ and, moreover, we can show that \eqref{salamon} holds for $M_m^n$ with $p=2n+1$.

\end{remark}

\begin{remark}\label{SL}
It follows from \cite{FP,Pu} that the complex solvmanifolds $(M^n_m,J_1)$  have trivial canonical bundle, via a left invariant holomorphic section. Therefore, following the lines of the proof of \cite[Corollary 3.3]{BDV} (see also \cite[Theorem 5.4]{GeTa}) and using the fact that the Obata connection $\nabla$ is flat,  we obtain that their holonomy group $\operatorname{Hol}(\nabla)$  is a discrete subgroup of $\operatorname{SL}(2n+1, \H)$.
\end{remark}

\medskip

\begin{example}\label{clifford}
Let $\g =\R\ltimes _A \R^{4n-1}$ be a hypercomplex almost abelian Lie algebra with $A$  given by \eqref{matrixL}. Then the corresponding Obata connection $\nabla$ is flat (Proposition \ref{Obata-flat}) and we can define a Lie algebra $T_{\nabla } \g:=\g \ltimes_\nabla \g  $ with the following Lie bracket:
\[ [(x,v), (y,w)] = ([x,y], \nabla _xw -\nabla _y v), \qquad x, y, v, w\in \g .
\]
The Jacobi identity is satisfied since $\nabla$ is flat (see \cite[\S3]{BD1}). Note that $\g\oplus \{ 0\} $  is a Lie subalgebra isomorphic to $\g$ and $\{0\}\oplus \g$ is an \textit{abelian} ideal.   It turns out that  $T_{\nabla } \g$ is again almost abelian with codimension one abelian ideal $\R^{4n-1}\oplus \g$. The action of $(e_0,0)$ on $T_\nabla \g$ is given by: 
\begin{equation}\label{A_mu} \ad_{(e_0,0)}= \begin{bmatrix}
0& & \\
 & A &  \\               
& & \tilde{A}
\end{bmatrix},\quad \text{ for } \quad 
\tilde{A}=\left[\begin{array}{c|ccc}
\mu &  & &\\
\hline
| & &&\\
 v_0& &A& \\
 | & && \end{array}\right],
\end{equation}
with respect to a basis $\{(e_j,0)\}\cup\{(0,e_j)\}$ of $T_\nabla \g$, where $\{e_j\}, \; j=0, \ldots, 4n-1$, is a basis of $\g$ and 
 $v_0$ is a column vector  as in Theorem \ref{characterization}. Note that $T_\nabla \g$ is unimodular if and only if $\mu + 2\tr{A} =0$; in particular, if $\g$ is unimodular (i.e. $\tr A =0$) then $T_\nabla \g$ is unimodular if and only if $\mu=0$.

If $\hcx$ denotes the hypercomplex structure on $\g$ then, since $\nabla J_\al =0$, \cite[Proposition 3.3]{BD1} implies that $\{J_\al^-\}$ defines a hypercomplex structure on $T_{\nabla } \g$ where:
\[ J_\al^- (x,v)= (J_\al x, - J_\al v ), \qquad x, v\in \g  .
\]
Moreover, $\nabla$ is torsion-free, hence the following endomorphism of $T_{\nabla } \g$ is a complex structure  (\cite[Theorem 4.1]{BD1}):
\[ K(x,v)= (v, -x), \qquad x,v\in \g .
\]
We point out that $K$ anticommutes with $J^- _1$ and $J^- _2$, and $J^- _1,J^- _2, K$ generate a Clifford algebra of order $3$ on $T_{\nabla } \g$.  The connection $\nabla$ on $\g$ induces a connection $\nabla ^1$ on $T_{\nabla } \g$ as follows:
\[ {\nabla}^1_{(x,v)}(y,w)=(\nabla_xy,\nabla_xw), \qquad x, y,v,w\in \g, 
\]
and $\nabla^1$ turns out to be a flat, torsion-free connection on $T_{\nabla} \g$  such that $\nabla ^1 K =0$, and also $\nabla ^1 J_1^-=\nabla ^1 J_2^- =0$. Set $T^1_{\nabla} \g=T_{\nabla} \g$ and for $l>1$, following \cite[\S4.2]{BD1}, we define inductively $ T_{\nabla}^l \g$ to be the tangent algebra of $(T_{\nabla}^{l-1} \g , \nabla ^{l-1})$, that is:
\[ 
T_{\nabla}^l \g=T_{\nabla ^{l-1}}(T_{\nabla}^{l-1} \g), 
\quad 
\nabla^l_{(u,v)}(u',v')= (\nabla ^{l-1}_u u', \nabla ^{l-1}_u v'), \quad u,u',v,v'\in T_{\nabla}^{l-1} \g.
\]
Then $\nabla^l$ is a flat, torsion-free connection on $T_{\nabla}^l \g$. Furthermore, $T_{\nabla}^l \g$ carries a Clifford structure of order $l+2$ such that all its generators are $\nabla^l$-parallel. The Lie algebras $T_{\nabla}^l \g$ are almost abelian and the Lie bracket on $T^l\g$ is determined by:
\[ \ad_{(e_0,0)}= \begin{bmatrix} 0&  \\
 & A \end{bmatrix} \oplus \tilde{A}\oplus \cdots \oplus \tilde{A} ,
\]
where $\tilde{A}$ from \eqref{A_mu} appears $2^l-1$ times. 

For the choice of $A$ given in Example \ref{hcx-non-HK}, the simply connected Lie group with Lie algebra $T_{\nabla}^l \g_n$ admits lattices for all $l\geq 1$. Therefore, we obtain solvmanifolds carrying Clifford structures of order $k$ for arbitrary $k\geq 2$.  Moreover, in the same manner as in Remark \ref{SL}, we can show that the holonomy group $\operatorname{Hol}(\nabla^l)$ is a discrete subgroup of $\operatorname{SL}(m, \H)$ for $m= 2^l(2n+1)$.
\end{example}

\ 

To end this section we exhibit examples of hypercomplex Lie groups arising from almost abelian Lie groups but  which are not  almost abelian themselves. More precisely, we consider the tangent bundle of a hypercomplex almost abelian Lie group.

\begin{example}
If $G$ is any real  Lie group and $\g$ is its Lie algebra then its tangent bundle $TG$ is diffeomorphic to $G\times \g$. Moreover, it admits a natural Lie group structure such that its Lie algebra is $T\g=\g\ltimes_{\ad}\g$, where  $\ad:\g\to \operatorname{End}(\g)$ is the adjoint representation of $\g$. The Lie bracket is given by \[
[(x,v), (y,w)] = ([x,y], [x,w] -[y, v]), \qquad x, y, v, w\in \g .
\]

If $\g$ carries a complex structure $J$ then, according to \cite{BD1}, $T\g$ admits a complex structure $J^+$ defined by $J^+(x,v)=(Jx,Jv)$.

If $\g=\R e_0\ltimes_A\R^d$ is an almost abelian Lie algebra, then the Lie bracket on $T\g$ is given by:
\begin{gather*}
    [(e_0,0),(x,0)]=(Ax,0), \quad  [(e_0,0),(0,v)]=(0,Av), \\
    [(v,0),(0,e_0)]=(0,-Av), \quad 
    [(x,0),(0,v)]=(0,0), 
\end{gather*}
for $x,v\in\R^d$. Therefore we can write $T\g=\R^2\ltimes(\R^d\oplus \R^d)$, where $\R^2$ is generated by $(e_0,0)$ and $(0,e_0)$, and the action of $\R^2$ on $\R^d\oplus \R^d$ is the following:
\[ A_1:=\ad_{(e_0,0)}|_{(\R^d\oplus \R^d)}=\left[\begin{array}{c|c}
 A & 0\\ \hline 0 & A
\end{array}\right],
\qquad A_2:=\ad_{(0,e_0)}|_{(\R^d\oplus \R^d)}=
\left[\begin{array}{c|c}
 0 & 0\\ \hline A & 0
\end{array}\right]. \]
Note that $T\g$ is unimodular if and only if $\g$ is unimodular. 

If $d=4n-1$ and $A$ is as in \eqref{matrixL}, so that   $\g$ carries a hypercomplex structure $\hcx$, then $T\g$ also carries a hypercomplex structure $\{J_\al^+\}$. Moreover, if $A$ is nilpotent  then $T\g$ is  nilpotent as well.   If, in addition, $G$ is simply connected and all the entries of $A$  are integers we have that $TG$ admits  lattices, giving rise to hypercomplex nilmanifolds. 
\end{example}

\ 

\end{document}